\documentclass[a4paper,reqno,oneside]{amsart}

\usepackage[
style=numeric,
url = false,
sorting=nty,
maxnames=99,
maxalphanames=5,
sortcites]{biblatex}

\usepackage[british]{babel}
\usepackage{csquotes}  
\numberwithin{equation}{section}
\usepackage{amssymb}
\usepackage[protrusion=true,expansion=true]{microtype}	
\usepackage{amsmath,amsfonts,amsthm} 
\usepackage[pdftex]{graphicx}	
\usepackage{url}
\usepackage{svg}
\usepackage[title,titletoc,page]{appendix} 
\usepackage{listings} 
\usepackage{color}
\usepackage{caption}
\usepackage{subcaption}
\usepackage{algorithm}
\usepackage{dsfont} 
\usepackage{bm}
\usepackage{upgreek}
\usepackage[normalem]{ulem}
\usepackage{enumerate}
\usepackage{mathtools}
\usepackage{lipsum}
\usepackage{adjustbox}
\usepackage{tikz}
\usepackage{circuitikz}
\usepackage{tcolorbox}
\usetikzlibrary{tikzmark,calc}
\usepackage{xargs}      
\usepackage{blkarray}

\usepackage{svg}
\usepackage{nicematrix}

\usepackage[T1]{fontenc} 
\usepackage{lmodern} 
\rmfamily 
\DeclareFontShape{T1}{lmr}{b}{sc}{<->ssub*cmr/bx/sc}{}
\DeclareFontShape{T1}{lmr}{bx}{sc}{<->ssub*cmr/bx/sc}{}

\usepackage[left=2.3cm,right=2.3cm,top=2.8cm,bottom=3.0cm,footskip=1.5cm,headsep=1cm,marginparwidth=1.9cm]{geometry}
\usepackage[%
pdfauthor={Yannick de Bruijn, Erik Orvehed Hiltunen},%
pdftitle={PHOTONIC BANDGAPS III},
pdfsubject={PHOTONIC BANDGAPS III},
pdfkeywords={Subwavelength resonances, evanescent modes, band gap, interface eigenmodes, Bloch theory, complex Brillouin zone, layer potentials}]{hyperref}

\usepackage{orcidlink}
\usepackage{fancyhdr}
\fancypagestyle{plain}{%
\fancyhead[C]{}
\fancyfoot[C]{}

}

\fancypagestyle{nosection}{%
\fancyhead[CE]{}
\fancyhead[CO]{}
\fancyfoot[CE,CO]{\thepage}
}
\numberwithin{equation}{section}		
\numberwithin{figure}{section}			
\numberwithin{table}{section}			

\newcommand{\vect}[1]{\boldsymbol{\mathbf{#1}}}

\DeclareMathOperator{\dd}{d\!}

\DeclareMathOperator{\C}{\mathbb{C}}

\newcommand{\N}{\mathbb{N}}

\newcommand{\R}{\mathbb{R}}

\newcommand{\Z}{\mathbb{Z}}

\renewcommand{\epsilon}{\varepsilon}

\newcommand{\Id}{\mathrm{Id}}

\renewcommand{\tilde}{\widetilde}

\renewcommand{\i}{\mathrm{i}}
\renewcommand{\d}{\,\mathrm{d}}

\newcommand{\iL}{\mathsf{L}}
\newcommand{\iR}{\mathsf{R}}
\newcommand{\iLR}{{\mathsf{L},\mathsf{R}}}

\newcommand{\sddots}{\raisebox{3pt}{\scalebox{.6}{$\ddots$}}}

\usepackage[noabbrev,capitalize]{cleveref}
\crefname{proposition}{Proposition}{Propositions}
\crefname{equation}{}{}

\newtheorem{theorem}{Theorem}[section]
\newtheorem{lemma}[theorem]{Lemma}
\newtheorem{proposition}[theorem]{Proposition}
\newtheorem{corollary}[theorem]{Corollary}

\theoremstyle{definition}
\newtheorem{definition}[theorem]{Definition}
\newtheorem{example}[theorem]{Example}

\crefname{assumption}{Assumption}{Assumptions}
\crefname{definition}{Definition}{Definitions}
\crefname{corollary}{Corollary}{Corollaries}
\crefname{enumi}{item}{items}

\begin{document}
\title[Topologically protected interface modes in multi-band damped lattice models]{Topologically protected interface modes in multi-band damped lattice models}

\author[Y. de Bruijn]{Yannick de Bruijn$^*$ \,\orcidlink{0009-0009-6194-0761}}
\address{\parbox{\linewidth}{Yannick de Bruijn\\
Department of Mathematics, University of Oslo, Moltke Moes vei 35, 0851 Oslo, Norway.
\\
\href{https://orcid.org/0009-0009-6194-0761}{orcid.org/0009-0009-6194-0761}}}
\email{yannicd@math.uio.no}

\author[E. O. Hiltunen]{Erik Orvehed Hiltunen \,\orcidlink{0000-0003-2891-9396}}
\address{\parbox{\linewidth}{Erik Orvehed Hiltunen\\
Department of Mathematics, University of Oslo, Moltke Moes vei 35, 0851 Oslo, Norway. \\
\href{http://orcid.org/0000-0003-2891-9396}{orcid.org/0000-0003-2891-9396}.}}
\email{erikhilt@math.uio.no}

\begin{abstract}
    Tridiagonal $k$-Toeplitz operators provide a natural framework for modelling one-dimensional $k$-periodic lattice systems. A fundamental connection is obtained between Coburn's lemma for tridiagonal $k$-Toeplitz operators and the existence of edge modes. We reveal that topological edge modes are characterised by the eigenvalues of the leading principal submatrix of the symbol function.
    A complete analysis of tridiagonal interface operators satisfying global inversion symmetry is then presented. These results are applied to finite one-dimensional $k$-periodic chains of damped resonators that satisfy both local and global inversion symmetry.
    Additionally, disordered tight-binding interface operators are shown to support a topologically robust zero-energy interface state. Numerical simulations are conducted to illustrate the theoretical findings.
\end{abstract}
\maketitle
\noindent
\textbf{Mathematics Subject Classification (MSC2020): }15A18 
15B05, 
35C20, 
81Q12.  

\vspace{3mm}

\section{Introduction}

Inspired by the study of damped resonator systems, this article tackles two unresolved problems. First, we provide a characterisation of the open limit of tridiagonal $k$-Toeplitz matrices. The open limit refers to the spectrum of a finite but large Toeplitz matrix whose size tends to infinity. This limit is particularly relevant for physical applications, where it is used to connect open and periodic boundary conditions in resonator chains. Our approach builds on Harold Widom’s results for the open limit of block Toeplitz matrices \cite{WIDOM1974284}. In the specific case of tridiagonal $k$-Toeplitz matrices, Widom’s open limit characterisation can be significantly simplified. 
The second contribution of this paper is a strikingly simple description of edge states in multi band lattice models, given by the eigenvalues of a principal minor of the symbol function. Coburn’s lemma for tridiagonal $k$-Toeplitz matrices \cite{ammari2024spectra,debruijn2025complexbandstructurelocalisation} is used to establish a bulk-boundary correspondence for tridiagonal $k$-Toeplitz operators, which guarantees the existence of edge modes and provides an optimal description of the spectrum of tridiagonal $k$-Toeplitz operators.
Topological protection of edge modes is a sought after property as it reveals that edge modes are highly robust against defects obtained from the bulk properties of the material. We reveal that if the symbol function of a tridiagonal $k$-Toeplitz matrix satisfies local inversion symmetry, the edge modes are topologically protected.

These new results will be used to investigate the spectra and in particular interface modes of tridiagonal $k$-Toeplitz interface matrices satisfying global inversion symmetry around the interface.
We distinguish two separate categories of interface modes. The first class comprises edge-induced interface modes that produce monopole-type interface states. 
The second class consists of eigenvalue-matched interface modes, which generate dipole-type interface states and can be seen as a discrete analogue of a matched impedance at the interface \cite{10.1098/rspa.2023.0533,xiao2014surface}.
Topological robustness \cite{Bal03082022, KaneTopological,fefferman2014topologically,lin2022mathematical,kane2005z}, it is shown, belongs exclusively to edge-induced interface modes, while the existence of eigenvalue-matched modes is dependent on the specific coupling at the interface. This stands in contrast to continuum models, where impedance matched interface modes are known to be topologically protected, whereas edge induced interface modes might not exist in continuum models \cite{10.1098/rspa.2023.0533, 10.1098/rspa.2022.0675, alexopoulos2024topologicalinterfacemodessystems,shapiro2022continuum}.

We illustrate our result in the case of a one-dimensional damped $k$-periodic subwavelength resonator chains obtained by merging two mirror symmetric materials together at an interface, creating a $k$-periodic SSH type structure \cite{PhysRevLett.42.1698, PhysRevB.95.115443, ammari2024exponentiallypreprint,craster2023asymptotic}. Although topological protection is well characterised in undamped systems \cite{Edge_Modes, 10.1098/rspa.2023.0533, ammari2020topologically}, no analogue of a topological quantification, such as the Zak phase \cite{PhysRevLett.62.2747}, appears to be known for damped systems.
This will be addressed in the present work, as we will generalise the topological protection of dimers established in \cite{ammari2020topologically} to $k$-periodic resonator chains, as well as demonstrate the protection of interface modes in damped $k$-periodic resonator arrangements satisfying both local and global inversion symmetry.

We conclude the study of edge induced interface modes by revealing an exceptionally high robustness with respect to disorder of a zero energy interface state. We simulate a one dimensional resonator chain with nearest-neighbour coupling and numerically confirm its robustness against variations in the coupling strengths \cite{SSHInterfaceZeroEnergy}.

The paper is organised as follows. Section \ref{Sec: Toeplitz Theory} reviews key results on the spectra of tridiagonal $k$-Toeplitz operators and derives an exact description of the open-limit spectra of tridiagonal $k$-Toeplitz matrices with complex entries. Furthermore, we examine the presence of edge modes and their protection by local inversion symmetry.
Section \ref{sec: Twofold Toeplitz} builds on this framework to characterise the spectra of interface Toeplitz operators. Necessary and sufficient conditions on the eigenvalues supporting a mid-gap interface mode are obtained. We further show that monopole modes arise from edge-induced interfaces, whereas dipole modes originate from a eigenvalue-matched interface.
In Section \ref{Sec: resonator chains}, these analytical results are applied to one-dimensional damped resonator chains and the differences between discrete and continuum models are elucidated in Section \ref{sec: Continuum Limit}. Section \ref{Sec: Aperiodic Dimers} considers a tight-binding Hamiltonian which exhibits an interface mode, robust under strong disorder. 
Finally, Section \ref{Sec: discussion} outlines potential applications and directions for future work. The code used for the numerical simulations in the figures is provided in Section \ref{Sec: Data availability}.

\section{Spectra of tridiagonal k-Toeplitz matrices with complex entries}\label{Sec: Toeplitz Theory}

We will start by recalling some fundamental results on tridiagonal $k$-Toeplitz operators as well as on the asymptotic spectra of large but finite Toeplitz matrices.
We define a \emph{tridiagonal $k$-Toeplitz operator} by the semi-infinite matrix,
\begin{equation}\label{equ:tridiagonalktoeplitzop1}
    \vect{A} := \left(\begin{smallmatrix}
        a_1 & b_1 &  & & & & &  \\
        c_1 & a_2 & b_2 & &  &\\
         & \sddots &  \sddots & \sddots  & & \\
         & & c_{k-2} & a_{k-1}& b_{k-1} & & \\
        & &  & c_{k-1} & a_k  & b_k&  &\\
        & & &  & c_{k} & a_1  & b_1&  &\\
        & & & &  & \sddots & \sddots&  &\sddots \\
    \end{smallmatrix}\right),
\end{equation}
where $a_i,b_i,c_i\in\C$ for $i\in\N$ and $\vect{A}_{ij}=0$ if $\vert i-j \vert > 1$ for $i,j\in\N$.
Any tridiagonal $k$-Toeplitz operator $\vect{A}$ can be reformulated as a tridiagonal block Toeplitz operator with $k \times k$ blocks following a $1$-periodic repetition, that is,
\begin{equation}\label{equ:tridiagonalktoeplitzop2}
    \vect{A} = \scalebox{0.8}{$
    \begin{pmatrix}
        \vect{A}_0 & \vect{A}_{-1} &  \\
        \vect{A}_{1} & \vect{A}_0 & \ddots  \\
         & \ddots & \ddots   \\
    \end{pmatrix}$}.
\end{equation}
The \emph{symbol} of a tridiagonal block Toeplitz operator \eqref{equ:tridiagonalktoeplitzop2}, and thus also of a tridiagonal $k$-Toeplitz operator \eqref{equ:tridiagonalktoeplitzop1}
is defined by the matrix-valued function
\begin{align}\label{def: symbol of operator}
    f:\mathbb C &\to \mathbb C^{k\times k}\nonumber\\
    z&\mapsto \vect{A}_{-1}z^{-1} + \vect{A}_0 + \vect{A}_1z.
\end{align}
Let $\mathbf{B}_0, \mathbf{B}_1 \in \C^{k-1 \times k-1}$ denote the leading and trailing principal sub-matrices of $\mathbf{A}_0\in \C^{k\times k}$, respectively, given by
\begin{equation}\label{def: B0}
    \mathbf{A_0} = \begin{pmatrix}
            \mathbf{B}_0
        &b_{k-1}\\
        c_{k-1} & a_k
    \end{pmatrix} = \begin{pmatrix}
        a_1 & b_1 \\
        c_1 &
            \mathbf{B}_1
         \\
    \end{pmatrix}.
\end{equation}
In the sequel of this paper, we shall refer to the Toeplitz operator $\vect{A}$, specified in equation \eqref{equ:tridiagonalktoeplitzop2}, by the notation $\vect{A} = \vect{T}(f)$.
We define the set of eigenvalues of the symbol function evaluated on the torus as
\begin{equation}\label{equ:detspectraformula1}
\sigma_{\mathrm{det}}(f) := \bigl\{\lambda\in \mathbb C: \det\bigl(f(z)-\lambda\bigr)=0 \text{ for some }  z\in \mathbb T \bigr\}.
\end{equation}
Following \cite[Theorem 2.4.]{ammari2024spectra}, the essential spectrum of the Toeplitz operator is given by \eqref{equ:detspectraformula1}, that is
\begin{equation}\label{eq: essential spectrum}
    \sigma_{\mathrm{ess}}\big(\mathbf{T}(f)\big) = \sigma_{\mathrm{det}}(f).
\end{equation}
Let us also introduce the region of non-trivial winding,
\begin{equation}\label{equ:windspectraformula1}
    \sigma_{\mathrm{wind}}(f):= \bigl\{\lambda \in \mathbb{C}\setminus \sigma_{\mathrm{det}}(f) :  \operatorname{wind}\bigl(\det\bigl(f(\mathbb T) - \lambda),0\bigr) \neq 0 \bigr\}.
\end{equation}
The spectrum of Toeplitz operators is characterised solely by its symbol function \cite[Theorem 2.7.]{ammari2024spectra}, that is, 
\begin{equation}\label{eq: spectrum Toeplitz operator}
    \sigma_\mathrm{det}(f) \cup \sigma_\mathrm{wind}(f)  \subseteq \sigma\bigl(\mathbf{T}(f)\bigr) \subseteq  \sigma_\mathrm{det}(f)  \cup \sigma_\mathrm{wind}(f)  \cup \sigma \bigl(\mathbf{B}_0\bigr),
\end{equation}
where $\mathbf{B}_0$ is the leading principal submatrix of $\mathbf{A}_0$ as in \eqref{def: B0}. The essential spectrum of the Toeplitz operator can alternatively be presented as follows.
Let $z_1(\lambda)$ and $z_2(\lambda)$ be the roots of the polynomial $z\operatorname{det}\big(f(z)-\lambda\Id\big) = 0$ and let us define the set of roots of equal magnitude as,
\begin{equation}\label{def: Gamma set}
    \Gamma := \big\{\lambda \in \C :|z_1(\lambda)| = |z_2(\lambda)| \big\}.
\end{equation}
The set $\Gamma$ is potentially difficult to compute, since it involves checking the magnitude of the roots $z_i(\lambda)$ for all $\lambda \in \C$. This can be overcome in the tridiagonal case using the following result.
\begin{lemma}\label{lemma: open limit essential spectrum}
    Let $f\in\C^{k\times k}$ be the symbol function of a tridiagonal $k$-Toeplitz operator, then it holds that 
    \begin{equation}\label{eq: Gamma Floquet}
        \Gamma = \bigcup_{\alpha \in [-\pi, \pi]} \sigma\big( f(re^{\i\alpha})\big),\quad \text{for }r = \sqrt{\prod_{i = 1}^k\left|\frac{b_i}{c_i}\right|}.
    \end{equation}   
\end{lemma}

\begin{proof}
    For tridiagonal $k$-Toeplitz operators, the determinant of the symbol corresponds to \cite[Appendix A]{ammari2024spectra}
    \begin{equation}\label{equ:detexpansion1}
        z\det\bigl(f(z)-\lambda\bigr) = (-1)^{k+1} z^2 \prod_{i = 1}^k c_i + g(\lambda)z + (-1)^{k+1}  \prod_{i = 1}^k b_i.
    \end{equation}
    As we are seeking roots of equal magnitude, it follows by Viète's Theorem that
    \begin{equation}
        |z_1(\lambda)||z_2(\lambda)| = |z_1(\lambda)z_2(\lambda)| = \left|\frac{(-1)^{k+1}  \prod_{i = 1}^k b_i}{(-1)^{k+1}  \prod_{i = 1}^k c_i}\right| =   \prod_{i = 1}^k\left|\frac{b_i}{c_i}\right| =:r^2
    \end{equation}
    from which it readily follows that $|z_1(\lambda)| = |z_2(\lambda)| = r$.
    Crucially, the magnitude of the roots is independent of $\lambda$ and therefore, the roots cluster along a circle of radius $r$. Consequently the set of possible $\lambda$ have to be roots of 
    \begin{equation}\label{eq: set of lambda}
    \operatorname{det}\big(f(r\mathbb{T}) - \lambda\Id\big) = 0
    \end{equation}
from which \eqref{eq: Gamma Floquet} is obtained immediately.
\end{proof}

In the remainder of this work, we will consider the case where $\mathbf{T}$ is symmetric. Note that, since we allow complex entries, we are not assuming $\mathbf{T}(f)$ to be Hermitian. Also note that any tridiagonal matrix can be reduced to a symmetric matrix by a similarity transform (as outlined in \cite[Section 3]{ammari2024spectra}).
In the case where $\mathbf{T}(f)$ is symmetric, the expression for the essential spectrum can be further simplified. Let $b_i = c_i \in \C, ~\forall i\in\{1,\dots,k\}$, then it holds,
\begin{equation}\label{def: A}
    A := (-1)^k\prod_{i = 1}^k b_i =  (-1)^k\prod_{i = 1}^k c_i \in \C.
\end{equation}

\begin{corollary}\label{cor: symmetric Toeplitz spectrum}
    Let $\mathbf{T}(f)$ be a symmetric tridiagonal $k$-Toeplitz operator, then 
    \begin{equation}\label{eq: simplified essential spectrum}
        \sigma_{\mathrm{det}}(f) = \bigl\{ \lambda \in \C:2A\cos(\alpha) + g(\lambda) = 0,~\alpha\in[0, 2\pi] \bigr\}
    \end{equation}
    and the spectrum is given by
    \begin{equation}\label{eq: spectrum complex}
        \sigma_{\mathrm{det}}(f) \subseteq \sigma\bigl(\mathbf{T}(f)\bigr) \subseteq\sigma_{\mathrm{det}}(f) \cup \sigma(\mathbf{B_0}),
    \end{equation}
    where $A$ is given by \eqref{def: A} and where $g(\lambda)$ is a polynomial of degree $k$ given by 
\begin{equation}\label{equ:defiglambda1}
        g(\lambda) = \det(\vect{A}_0-\lambda)  -b_k^2 p(\lambda)  ,
    \end{equation}
    where
    \begin{equation}
    p(\lambda) = \begin{cases}
                    0, & k=1, \\
                    1, & k=2,\\
                    \left\lvert\begin{smallmatrix}
                            a_2 - \lambda & b_2 & 0 & \dots & 0 \\
                            b_3 & a_3 - \lambda & b_3 & \ddots & \vdots \\
                            0 & \ddots & \ddots & \ddots & 0 \\
                            \vdots & \ddots & \ddots & \ddots & b_{k-2} \\
                            0 & \dots & 0 & b_{k-1} & a_{k-1} - \lambda \\
                    \end{smallmatrix}\right\rvert, & k \geq 3.
                \end{cases}
    \end{equation}
\end{corollary}

\begin{proof}
    Since the operator is symmetric, \eqref{equ:detspectraformula1} reduces to,
    \begin{equation}\label{eq: collapsed symbol}
    \det\bigl(f(e^{-\i\alpha})-\lambda \bigr) = A e^{-\i\alpha} + A e^{\i\alpha} + g(\lambda) = 2A \cos(\alpha) + g(\lambda).
\end{equation}
Moreover, from \eqref{eq: collapsed symbol} it is not hard to see that the symbol function evaluated on the torus, i.e. $\alpha \in [0, 2\pi]$ does not have any interior points, so
\begin{equation}
    \sigma_{\mathrm{wind}}(f) = \emptyset.
    \end{equation}
    The inclusion in \eqref{eq: spectrum complex} is a direct consequence of \eqref{eq: spectrum Toeplitz operator} asserted in \cite[Theorem 2.7]{ammari2024spectra}.
\end{proof}

Therefore, in the symmetric case, up to possibly $k-1$ points entailed by $\sigma(\mathbf{B}_0)$, the spectrum of a Toeplitz operator is given by $\sigma_{\mathrm{det}}(f)$, which in turn is given by the roots of a parametrised equation
\begin{equation}
    A \cos(\alpha) + g(\lambda)  = 0, ~\text{for } \alpha \in [0, 2\pi].
\end{equation}
Note that the function $g(\lambda)$ is explicitly given by \eqref{equ:defiglambda1} and is a polynomial of degree $k$, so the set $\sigma_{\mathrm{det}}(f)$ is the set of roots parametrised in $\alpha$ and can therefore easily be computed numerically. From this it also becomes clear that the set $\sigma_{\mathrm{det}}(f)$ is composed of a collection of algebraic curves in the complex plane.
\begin{lemma}\label{lemma: open limit and essential spectrum}
    Let $f\in \C^{k\times k}$ be the symbol function of a symmetric Toeplitz operator $\mathbf{T}(f)$, then it hold that 
    \begin{equation}
        \Gamma = \sigma_{\mathrm{ess}}\big(\mathbf{T}(f)\big).
    \end{equation}
\end{lemma}

\begin{proof}
    Since $\mathbf{T}(f)$ is symmetric, it follows from \eqref{eq: Gamma Floquet} that $r = 1$, in which case 
    \begin{equation}
        \Gamma = \bigcup_{\alpha \in [-\pi, \pi]} \sigma\big( f(e^{\i\alpha})\big) =  \bigl\{\lambda\in \mathbb C:\det\bigl(f(z)-\lambda\bigr)=0,\ \exists  z\in \mathbb T \bigr\} = \sigma_{\mathrm{ess}}\big(\mathbf{T}(f)\big)
    \end{equation}
    where in the last equality, we used \eqref{eq: essential spectrum}, which completes the proof of the result.
\end{proof}

We now turn to the point spectrum of $\mathbf{T}(f)$ comprising up to $k-1$ points entailed by $\sigma(\mathbf{B}_0)$.

\begin{definition}[Edge states]\label{def: edge state}
We say that $\mathbf{v} \in \ell^2(\N)\setminus\{\mathbf{0}\}$ is an \emph{edge mode} of $\mathbf{T}(f)$ with an \emph{edge eigenvalue} $\lambda\in \C \setminus\sigma_{\mathrm{ess}}\big(\mathbf{T}(f)\big)$ if $\mathbf{T}(f)\mathbf{v} = \lambda \mathbf{v}$, where $\mathbf{v}$ decays exponentially away from the index $0$. We denote by $\sigma_{\mathrm{edge}}\big(\mathbf{T}(f)\big)\subseteq\sigma(\mathbf{B_0})$, the set of edge eigenvalues of $\mathbf{T}(f)$.
\end{definition}
To analyse the occurrence of edge states, it is essential to first consider the following result from \cite[Theorem 2.3]{ammari2024spectra}, which extends Coburn's lemma to the tridiagonal $k$-Toeplitz setting.

\begin{theorem}[Coburn's lemma, tridiagonal k-Toeplitz version]\label{thm: coburns lemma tridiagonal}
   Let $f \in \C^{k\times k}(\mathbb{T})$ be the symbol of a tridiagonal $k$-Toeplitz operator such that $\det\bigl(f(z)\bigr)$ does not vanish identically on $\mathbb{T}$. Then one of the following statements holds:
   \begin{enumerate}[(i)]
       \item $\mathbf{T}(f)$ has a trivial kernel.
       \item $\mathbf{T}(f)$ has a dense range.
       \item The leading $(k-1)\times(k-1)$ principal submatrix of $\mathbf{A}_0$, that is $\mathbf{B}_0$, has a non-zero kernel. In particular, there exists some $z_0 \in \C\cup\{\infty\}$ such that $\ker(z_0^{-1}\mathbf{A}_{-1}+\mathbf{A}_0)\cap\ker(\mathbf{A}_1)\neq \vect 0$.
   \end{enumerate}
\end{theorem}

Coburn's lemma for tridiagonal $k$-Toeplitz operators yields a characterisation on the existence of edge states can therefore be seen as the analogue of the bulk-boundary correspondence \cite{10.1098/rspa.2022.0675} for tridiagonal $k$-Toeplitz operators.

\begin{theorem}\label{cor: coburn eigenvector}
    Let $f\in \C^{k\times k}$ be the symbol function of a tridiagonal $k$-Toeplitz operator, then
    \begin{enumerate}[(i)]
        \item If $\lambda \in \sigma(\mathbf{B}_0) \setminus \Gamma$, then there exists a $z(\lambda) \in \C$ and a vector $\mathbf{v}\in\C^k$ such that
        \begin{equation}\label{eq: B0 Bloch eigenvector}
            \Big[f\big(z(\lambda)\big) -\lambda \Id\Big] \mathbf{v} = 0,\quad \text{with }\mathbf{v}_k =0.
        \end{equation}
        \item If $\lambda \in \sigma(\mathbf{B}_0) \setminus \Gamma$ and $z(\lambda)$ from \eqref{eq: B0 Bloch eigenvector} satisfies $|z(\lambda)|>1$, then  $\lambda\in \sigma_{\mathrm{edge}}\big(\mathbf{T}(f)\big) \subseteq \sigma\big(\mathbf{T}(f)\big)$.
        \item The spectrum of $\mathbf{T}(f)$ is given by,
        \begin{equation}\label{eq: exact Toeplitz spectrum}
            \sigma\big(\mathbf{T}(f)\big) = \sigma_{\mathrm{det}}(f) \cup \sigma_{\mathrm{edge}}\big(\mathbf{T}(f)\big).
        \end{equation}
    \end{enumerate}
\end{theorem}

\begin{proof}
    To show \textit{(i)} by Theorem \ref{thm: coburns lemma tridiagonal}, it follows that, if $\lambda \in \sigma(\mathbf{B}_0)$, then there is a vector $\mathbf{v}$ and a root $z\in \C$ such that $\mathbf{v}\in\operatorname{ker}\big(z^{-1}\mathbf{A}_{-1} + (\mathbf{A}_0-\lambda\Id)\big)\cap\operatorname{ker}(\mathbf{A}_1)$, but then $\vect{v}$ also satisfies
    \begin{equation}
        \big[f(z)- \lambda\Id\big]\vect{v} =\big[z^{-1}\mathbf{A}_{-1} +(\mathbf{A}_0-\lambda\Id)\big]\vect{v} + z\mathbf{A}_1\vect{v} = 0.
    \end{equation}
    So $\mathbf{v}$ is an eigenvector of $f(z)$ corresponding to the eigenvalue $\lambda$. Since $\mathbf{A}_1 \in \mathbb{C}^{k \times k}$ has non-zero entries only in the top right corner, it follows that $\mathbf{v}_k = 0$. 
    
    To show \textit{(ii)}, it is not hard to see that if $|z(\lambda)|>1$ then, 
    \begin{equation}\label{eq: quas extension}
        \mathbf{w} = \big(\mathbf{v}^\top, z(\lambda)^{-1}\mathbf{v}^\top, z(\lambda)^{-2}\mathbf{v}^\top, \dots\big) \in \ell^2(\N)
    \end{equation}
    satisfies $\mathbf{T}(f)\mathbf{w} = \lambda \mathbf{w}$, from which it follows that $\lambda \in \sigma_{\mathrm{edge}}\big(\mathbf{T}(f)\big)\subseteq \sigma\big(\mathbf{T}(f)\big)$.
    
    To show the spectral equality in \textit{(iii)}, let $\lambda \in  \sigma\big(\mathbf{T}(f)\big)\setminus\sigma_{\mathrm{det}}(f)$. But as only the quasiperiodically extended Bloch eigenvector $\mathbf{v}$ with $\mathbf{v}_k = 0$ is an $\ell^2$ eigenvector of $\mathbf{T}(f)$, it follows from \textit{(ii)} that $\lambda \in \sigma_{\mathrm{edge}}\big(\mathbf{T}(f)\big)$, which completes the proof of the theorem.
\end{proof}

Using our novel insights on tridiagonal $k$-Toeplitz operators we may greatly simplify the results of Widom \cite{WIDOM1974284} on the asymptotic spectrum of block Toeplitz matrices, which will be presented in greater detail in Section \ref{sec: open limit}.
Let us introduce the following set
\begin{equation}\label{def: G0}
    G_0 := \big\{ \lambda \in \C \setminus \Gamma :C_0(\lambda) = 0 \big\},
\end{equation}
with
\begin{equation}\label{def: C0}
    C_0(\lambda) := \operatorname{det}\Big(\mathbf{T}\big((f-\lambda\Id)^{-1}\big)\mathbf{T}(f- \lambda \Id)\Big).
\end{equation}
The function $C_0$ may alternatively be defined as 
\begin{equation}\label{eq: alternative C_0}
    C_0(\lambda) = \operatorname{det}\left(\frac{1}{2\pi \i} \int_{\gamma}\big(f(z)-\lambda\Id\big)^{-1} \frac{\d z}{z}\right),\quad \lambda\in\C\setminus\Gamma,
\end{equation}
where $\gamma$ is a counter-clockwise oriented closed Jordan curve enclosing the origin and the root of smallest magnitude $z_1(\lambda)$ but not the root $z_2(\lambda)$.
Widom already observed in \cite{WIDOM1974284} that in the $1$-Toeplitz case one must have $G_0 = \emptyset$. For the more general tridiagonal $k$-Toeplitz matrices, Coburn's lemma gives a valuable insight on structure of the set $G_0$.

\begin{theorem}\label{thm: edge mode characcterisation}
    Let $f\in \C^{k\times k}$ be the symbol function of a tridiagonal $k$-Toeplitz matrix, then it holds that 
    \begin{equation}\label{eqref: inclusion G0}
        G_0 \subseteq \sigma(\mathbf{B_0})\cup \sigma(\mathbf{B}_1),
    \end{equation}
    where $\mathbf{B}_0,\mathbf{B}_1\in \C^{k-1\times k-1} $ are the leading and trailing principal submatrices of $\mathbf{A}_0$ receptively as in \eqref{def: B0}.
\end{theorem}

\begin{proof}
    Let $\lambda \in G_0$ then it holds by \eqref{def: C0} that $\operatorname{det}\big(\mathbf{T}((f-\lambda\Id)^{-1})\mathbf{T}(f- \lambda \Id)\big) = 0.$
    In other words, $\mathbf{T}((f-\lambda\Id)^{-1})\mathbf{T}(f- \lambda \Id)$ is not invertible. This implies that either $\mathbf{T}(f-\lambda\Id)$ or $\mathbf{T}((f-\lambda\Id)^{-1})$ is not invertible. We discuss both cases.
    First note that $\mathbf{T}(f-\lambda\Id) = \mathbf{T}(f)-\lambda\Id$, since $\mathbf{T}(f)-\lambda\Id$ is not invertible, it must hold that $\lambda \in \sigma(\mathbf{T})$. Since by assumption $\lambda \not\in \Gamma$, it follows from Corollary \ref{cor: symmetric Toeplitz spectrum} that $\lambda \in \sigma(\mathbf{B}_0)$.
    In the second case, we have to introduce the associated Toeplitz operator $\tilde{f}(z) = W_kf(z^{-1})W_k$ and $(W_k)_{ij} = \delta_{i(k+1-j)}$, then it holds that $\mathbf{T}\big(\tilde{(f-\lambda)}\big)$ is invertible if and only if $\mathbf{T}((f-\lambda\Id)^{-1})$ is invertible \cite[Section 6.2]{LargeTruncatedToeplitz}. This implies that $\lambda \in \sigma(\mathbf{T}(\tilde{f}))$ so $\lambda \in \sigma(\mathbf{B}_1)$, which completes the proof.
\end{proof}

Theorem~\ref{thm: edge mode characcterisation} is especially convenient for the numerical computation of the set $G_0$. Rather than checking the condition $C_0(\lambda) = 0$ for all $\lambda \in \mathbb{C} \setminus \Gamma$, it suffices to verify the condition $C_0(\lambda)=0$ at at most $2k-2$ discrete values of $\lambda$, explicitly given by $\sigma(\mathbf{B_0})\cup \sigma(\mathbf{B}_1)$. We emphasise that the converse inclusion in \eqref{thm: edge mode characcterisation} does not hold in general. A counterexample is the symbol function
\begin{equation}
    f(z) = 
    \begin{pmatrix}
        0 & 1 + \frac{1}{2}z \\
        1 + \frac{1}{2}z^{-1}&1\\
    \end{pmatrix}.
\end{equation}
Then $0 \in \sigma(\mathbf{B}_0)$, as well as $1\in \sigma(\mathbf{B}_1)$, but $C_0(0) = 1$ and $C_0(1) = 1$, so $0,1 \not\in G_0$.
An alternative characterisation of edge states, compared to Theorem \ref{cor: coburn eigenvector} may be achieved using the set $G_0$ by the following result.

\begin{theorem}\label{thm: existence of edge modes}
    If $\lambda \in \C$ is an edge eigenvalue of $\mathbf{T}(f)$, if and only if $\lambda \in G_0\cap\sigma(\mathbf{B}_0)$.
\end{theorem}

\begin{proof}
    Assume $\lambda \in \C$ is an edge eigenvalue in other words $\mathbf{T}(f)-\lambda\Id$ has non-trivial kernel $\mathbf{v}$, so it is not invertible, consequently $\operatorname{det}\big(\mathbf{T}\big((f-\lambda\Id)^{-1}\big)\mathbf{T}(f- \lambda \Id)\big) = 0$ so $\lambda \in G_0$. Following Definition \ref{def: edge state}, $\lambda \in \sigma(\mathbf{B}_0)$, therefore $\lambda \in G_0 \cap \sigma(\mathbf{B}_0)$.
    Assume now that $\lambda \in G_0\cap\sigma(\mathbf{B}_0)$, using that $\lambda \in \sigma(\mathbf{B}_0)$ it follows by the same method as in the proof of Theorem \ref{thm: edge mode characcterisation} it must hold that $\mathbf{T}(f)-\lambda\Id$ has non-trivial kernel. Since $\lambda \in G_0$ the eigenmode $\mathbf{v} \in \ell^2$ it follows that $|z(\lambda)|>1$, so by Theorem \ref{cor: coburn eigenvector} \textit{(ii)} $\lambda \in \sigma_{\mathrm{edge}}\big(\mathbf{T}(f)\big)$ which completes the proof. 
\end{proof}

\begin{example}\label{Example: existence of edge modes}
To illuminate the connection between the bulk-boundary correspondence established by the Zak phase and the characterisation of edge modes in Theorems~\ref{cor: coburn eigenvector} and \ref{thm: existence of edge modes}, we examine the case of a two-band model. The symbol function is defined as follows,
\begin{equation}
    f(z) = \begin{pmatrix}
        a_1&b_1 + b_2z\\
        b_1 + b_2/z &a_2
    \end{pmatrix}~\in \C^{2\times 2},    
\end{equation}  
By Theorem \ref{cor: coburn eigenvector} $\lambda \in \sigma(\mathbf{B}_0)= \{a_1\}$ is an edge mode provided that $|z|>1$. Following Theorem \ref{cor: coburn eigenvector} \textit{(i)}, $\mathbf{v} = [1, 0]^\top$ and a direct computation yields that,
\begin{equation}
    \mathbf{v}_1(b_1 +b_2/z) + 0a_2 =0 \Leftrightarrow z = -b_2/b_1,
\end{equation} 
and $\lambda$ is a edge mode provided that $|b_1| < |b_2|$.
Let us now consider the edge mode characterisation provided in Theorem \ref{thm: existence of edge modes}. The following inverse is explicitly given by,
\begin{equation}
    \big(f(z)-\lambda \Id\big)^{-1} = \frac{1}{(a_1-\lambda)(a_2-\lambda)-b_1^2 + b_2^2 + b_1 b_2\left(z + \frac{1}{z}\right)}\begin{pmatrix}
        a_1-\lambda& -b_1-b_2z \\
        -b_1 - b_2/z & a_2-\lambda
    \end{pmatrix}.
\end{equation}
Taking $\lambda \in \sigma(\mathbf{B}_0) = \{a_1\}$, computing entry-wise contour integrals, and then taking the determinant yields,
\begin{align}
    C_0(a_1) &= \operatorname{det}\left(\frac{1}{2\pi \i} \int_{\gamma}\big(f(z)-a_1\Id\big)^{-1} \frac{\d z}{z}\right) = \begin{cases} 1/b_1^2 & \text{if } |b_1| > |b_2|, \\ 0 & \text{if } |b_1| < |b_2|. \end{cases}
\end{align}
Both Theorems are consistent with the bulk boundary correspondence via the Zak phase \cite{Edge_Modes}, where it was shown that the Zak phase is non-trivial provided that $|b_1| <|b_2|$.
\end{example}

In the sequel, we will examine the topological properties of edge modes. Topological classifications are of particular use in physics as they demonstrate robustness with respect to perturbations. 
\begin{definition}[Topologically protected]\label{def: topologically protected}
An edge mode to the eigenvalue $\lambda$ is said to be topologically protected if it is robust under continuous perturbations of the symbol function that preserve local inversion symmetry, provided that the band gap containing $\lambda$ does not close.
\end{definition}

Since the topological classification of edge modes differs markedly between the Hermitian and complex symmetric cases, each is treated separately in the subsequent sections.

\subsection{Edge modes in real symmetric systems} Toeplitz matrices with real-valued and symmetric entries play a central role for the applications to lattice models, see Section~\ref{Sec: resonator chains}. Because Hermitian Toeplitz operators have a real spectrum, the spectral gap can be defined as $\Gamma := \R \setminus \sigma_{\mathrm{ess}}$, which consists of a collection of open disjoint intervals. Edge modes can be ordered within each such band gap, and this implies that there can be at most one edge-induced interface mode $\lambda \in \sigma(\mathbf{B}_0)$ in any given spectral gap.
\begin{lemma}\label{lem: uniqueness edge induced interface}
    Let $\lambda \in \sigma(\mathbf{B}_0) \setminus \sigma_{\mathrm{ess}}\big(\mathbf{T}(f)\big)$ be an eigenvalue of a Hermitian Toeplitz operator $\mathbf{T}(f)$, then $\lambda$ is the unique edge eigenvalue in the specific spectral gap situated between two spectral bands.
\end{lemma}

\begin{proof}
    The essential spectrum is given by \eqref{eq: essential spectrum}, that is the union of eigenvalues of the symbol function $f(z)\in\C^{k\times k}$ evaluated on the torus and consists of up to $k$ disjoint intervals. Then $\mathbf{B}_0$ is the leading $k-1\times k-1$ submatrix of $f(z)$, in particular the entries of $\mathbf{B}_0$ no longer depend on $z$. By the Cauchy interlacing theorem, the $k-1$ eigenvalues of $\mathbf{B}_0$ interlace the $k$ eigenvalues of $f(z),~\forall z \in \mathbb{T}$, that is
    \begin{equation}\label{eq: interlacing B0}
        \lambda_1\big(f(z)\big) \leq \mu_1(\mathbf{B}_0) \leq \lambda_2\big(f(z)\big)\leq \dots \leq \lambda_{k-1}\big(f(z)\big)\leq\mu_{k-1}(\mathbf{B}_0)  \leq \lambda_k\big(f(z)\big), \quad \forall z \in \mathbb{T}.
    \end{equation}
    The uniqueness follows immediately by the interlacing property, which completes the proof of the result.
\end{proof}

In order to derive a topological quantification of edge modes, we impose some symmetry constraints on the symbol function. Let us define the exchange matrix $(\mathbf{J})_{i,j} = \delta_{i, k+1-j}$, that is a matrix with constant $1$ on the antidiagonal, then the symbol function $f(z) \in \C^{k\times k}$ is called \emph{persymmetric} if $f(z) = \mathbf{J} f(z)^\top \mathbf{J}$.
The matrix is called  \emph{centrosymmetric} if it additionally holds that $f(z) = f(z)^\top$. Since $z\in\mathbb{T}$ and the operator $\mathbf{T}(f)$ is symmetric, the symbol function $f(z)$ is centrosymmetric if and only if $z = \pm 1$. The eigenvectors $\mathbf{v}$ of centrosymmetric matrices are symmetric or skew symmetric \cite{CANTONI1976275}, that is $\mathbf{v}_1 = \pm \mathbf{v}_k$.

\begin{theorem}\label{thm: inversion symmetry topological}
    Suppose that $f(z) \in \mathbb{R}^{k \times k}$ is persymmetric, then every edge mode $\lambda \in \sigma(\mathbf{B}_0)\setminus\sigma_{\mathrm{ess}}\big(\mathbf{T}(f)\big)$ is topologically protected.
\end{theorem}

\begin{proof}
    Suppose for contradiction that there exists $\lambda_0 \in \sigma(\mathbf{B}_0)$ that merges into the spectral bands without the spectral gap closing. This merging must occur at a band extremity, i.e. at an eigenvalue $\lambda_{\mathrm{ess}}$ of $f(\pm 1)$.
    Let $f(\pm 1)\mathbf{w}=\lambda_{\mathrm{ess}}\mathbf{w}$ then since by assumption $f(\pm 1)$ is centrosymmetric, the eigenvector $\mathbf{w}\in \R^k$ must satisfy $\mathbf{w}_1 = \pm \mathbf{w}_k$. In addition, it must hold that both entries are not equal to $0$ as otherwise because of the structure of $f(z)$ and the eigenvectors are given by a two-term recurrence, it would have to hold that $\mathbf{w}\equiv \mathbf{0}.$
    Let us now suppose that $\lambda_0 \in \sigma(\mathbf{B}_0)$ is an eigenvalue in the bandgap, by Theorem \ref{cor: coburn eigenvector}, the last entry of the associated eigenvector must be $0$. Since the essential spectrum $\sigma_{\mathrm{ess}}\big(\mathbf{T}(f)\big)$ is always closed, the spectral gap given by $\R\setminus \sigma_{\mathrm{ess}}\big(\mathbf{T}(f)\big)$ is open.
    We now move to the limit $\lambda_0 \to \lambda_{\mathrm{ess}}$.
    By assumption that the spectral gap remains open and by Lemma \ref{lem: uniqueness edge induced interface} both $\lambda_{\mathrm{ess}}$ and $\lambda_0$ are simple. It is well known that the Floquet parameter $z$ is continuous in the eigenvalue $\lambda$, therefore $f\big(z(\lambda)\big)$ is a continuous transformation in $\lambda$. By Lemma \ref{lem: uniqueness edge induced interface} the eigenvalue $\lambda_0$ is simple, the eigenvector entries vary continuously. 
    Therefore, since $\mathbf{v}(\lambda_0)_k = 0$ for all $\lambda_0 \in \sigma(\mathbf{B}_0) \setminus \sigma_{\mathrm{ess}}\!\left(\mathbf{T}(f)\right)$, continuity forces $\mathbf{v}(\lambda_{\mathrm{ess}})_k = 0$.
    This contradicts the conclusion $\mathbf{w}_k \neq 0$ established above, completing the proof.
\end{proof}

\subsection{Edge Modes in the Complex Symmetric Setting}

An analogue of Lemma~\ref{lem: uniqueness edge induced interface} in the complex symmetric setting cannot be obtained by the same approach, for the following reasons. First, the spectral gap $\Gamma := \mathbb{C} \setminus \sigma_{\mathrm{ess}}$ is now a connected set of infinite two-dimensional Lebesgue measure, so the interlacing property no longer holds. Second, the argument used in the proof of Theorem~\ref{thm: inversion symmetry topological} exploited the fact that, in the real symmetric case, edge modes can only collide with eigenvalues at $f(\pm 1)$ and spectral bands can only touch at $z = \pm 1$. Neither of these constraints persists in the complex setting. Edge modes may instead collide at the eigenvalues of $f(e^{\pm \i \alpha})$ for any $\alpha \in [0, 2\pi]$, and spectral bands may likewise touch at any such $z$. This is illustrated in Figure~\ref{Fig: Trimer Gap closing}. Consequently, the symmetry argument centred on the distinguished points $f(\pm 1)$ is no longer applicable.
Nevertheless, for a dimer symbol with complex-valued entries, the following result holds.

\begin{proposition}\label{prop: complex dimer topology}
    Let $f(z)\in \C^{2\times 2}$ be persymmetric, then the edge mode $\lambda \in \sigma(\mathbf{B}_0) = a_1$ is topologically protected.
\end{proposition}

\begin{proof}
    For a $2\times 2$ matrix, both the eigenvalues and eigenvectors have an  explicit closed form. Let us assume that $\lambda \in \sigma(\mathbf{B}_0)$ approaches the continuous spectrum, without the bandgap closing. By Theorem \ref{cor: coburn eigenvector}, for every eigenvector associated to $\lambda \in \sigma(\mathbf{B}_0)\setminus\sigma_{\mathrm{ess}}\big(\mathbf{T}(f)\big)$ must satisfy $\mathbf{v}=[1, 0]^\top$. By the assumption that the spectral band does not close, the eigenvalues of $f(\mathbb{T})$ stay simple, therefore the eigenvectors are continuous under continuous perturbations, by continuity it follows that $\mathbf{v}=[1,0]^\top$ also for $\lambda \in \sigma(\mathbf{B}_0)$. A direct computation yields that
    \begin{equation}
        \begin{pmatrix}
            a_1-\lambda & b_1 + b_2z \\
            b_1 + b_2/z & a_1-\lambda
        \end{pmatrix}\begin{pmatrix}
            1 \\
            0
        \end{pmatrix} = \mathbf{0}
    \end{equation}
    which holds as long as $z = - b_2/b_1$. Suppose now the edge mode is approaching the essential spectrum in the limit, then the Floquet parameter has to satisfy $|z|\to 1$, in other words $-b_2/b_1 = e^{\i\alpha}$.
    Hovewer, this is precisely the condition for bandgap closing. To see this, it is enough to understand that bandgap closing occurs when the symbol function has an eigenvalue of multiplicity $2$ for some $z\in\mathbb{T}$. In other words, the double eigenvalue has to satisfy
    \begin{equation}
        \operatorname{det}\big(f(z) -\tilde{\lambda}\big) = 0\quad \text{ and }\quad \frac{\d}{\d \lambda}\operatorname{det}\big(f(z) -\lambda\big)\Big|_{\lambda = \tilde{\lambda}} = 0.
    \end{equation}
 In particular, it holds by \eqref{eq: simplified essential spectrum} that
    \begin{equation}
        \frac{\d}{\d \lambda}\operatorname{det}\big(f(z) -\lambda\big) = \frac{\d}{\d \lambda}g(\lambda) = \frac{\d}{\d \lambda}\big(\lambda^2 - 2a_1\lambda + a_1^2 -b_1^2 -b_2^2\big) = 2\lambda -2a_1.
    \end{equation}
 Consequently, it must hold that $\tilde{\lambda} = a_1$, it remains to check the conditions when $\tilde{\lambda} = a_1$ is an eigenvalue of $f(z)$.
 \begin{equation}\label{eq: confluent root}
     \operatorname{det}\big(f(z) -\tilde{\lambda}\big) = b_1^2 + b_2^2 + b_1b_2(1/z + z) = 0
 \end{equation}
 It is easy to see that \eqref{eq: confluent root} holds precisely when $b_2 = -b_1e^{\i\alpha}$.
 To conclude, we have shown that the conditions for an edge mode vanishing into the essential spectrum are equivalent to the bandgap closing at this specific point. The edge mode is thereby topologically protected, which completes the proof.
\end{proof}
A particular property of complex symmetric dimer arrangements is the one-to-one correspondence with the edge mode vanishing and the bandgap closing, therefore it still makes sense to speak of topologically protected edge modes.
\begin{figure}[htb]
    \centering
    \subfloat[][Computation performed for $t = 0.1$.]%
    {\includegraphics[width=0.45\linewidth]{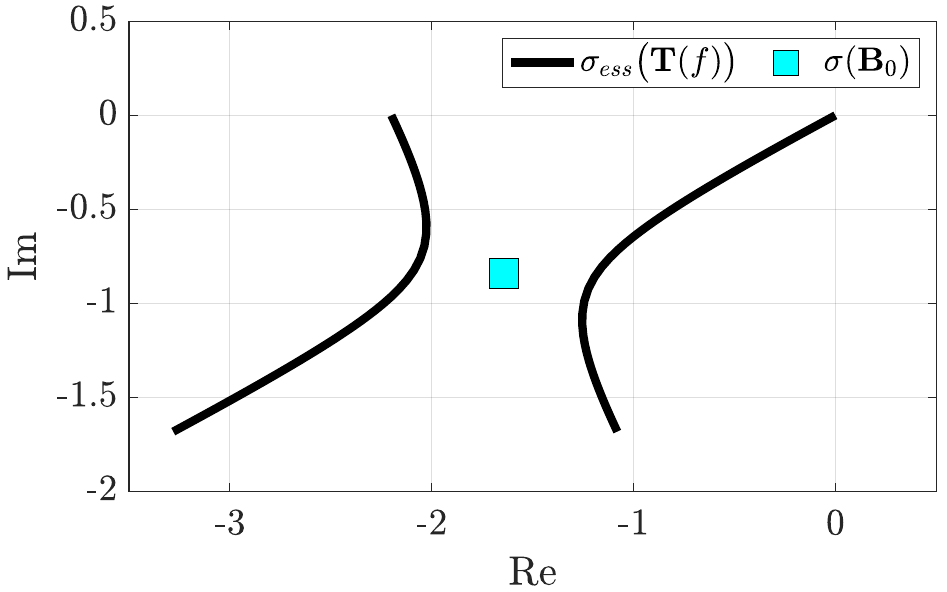}}\quad
    \subfloat[][Computation performed for $t=0$.]%
    {\includegraphics[width=0.45\linewidth]{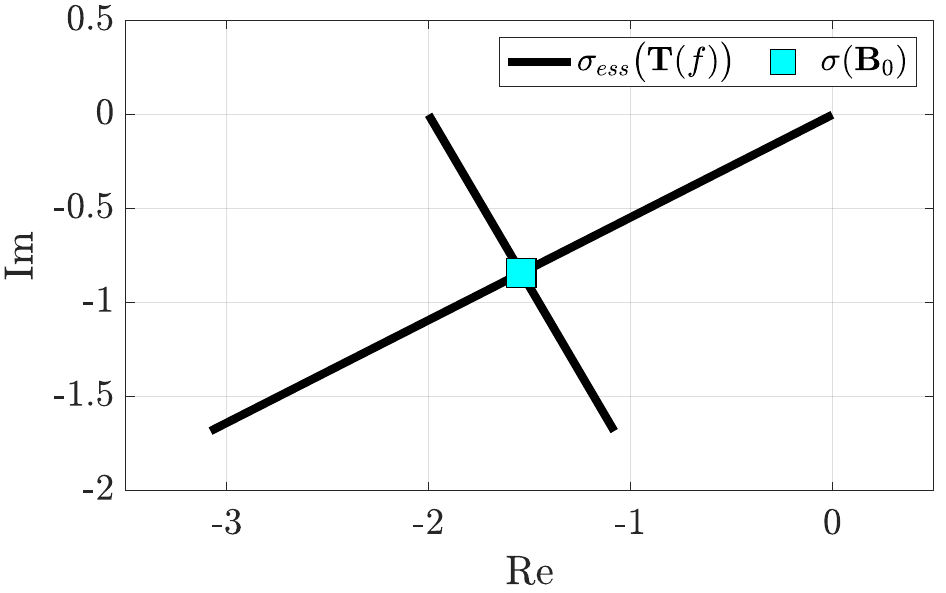}}
    \caption{Computation performed for $a_1 = a_2 = -2e^{\i}, b_1 =e^{\i}$ and $ b_2 = b_1e^{-\i} + t$.
    In a complex symmetric dimer arrangement, as asserted by Proposition \ref{prop: complex dimer topology} the gap closing condition is equivalent to the condition of the edge mode colliding with the essential spectrum.}
    \label{Fig: Complex Dimer Gap closing}
\end{figure}

We would like to emphasise that a similar result as in Proposition \ref{prop: complex dimer topology} does not hold for symbol functions $f\in \C^{k\times k}$ for $k\geq 3$. A counterexample is presented in Figure \ref{Fig: Trimer Gap closing}. Nevertheless, a generalisation of Proposition \ref{prop: complex dimer topology} to $k\geq 3$ is presented under much stronger assumptions in Proposition \ref{prop: top protected k geq 3}.

\begin{figure}[htb]
    \centering
    \subfloat[][Computation performed for $t = 1$.]%
    {\includegraphics[width=0.45\linewidth]{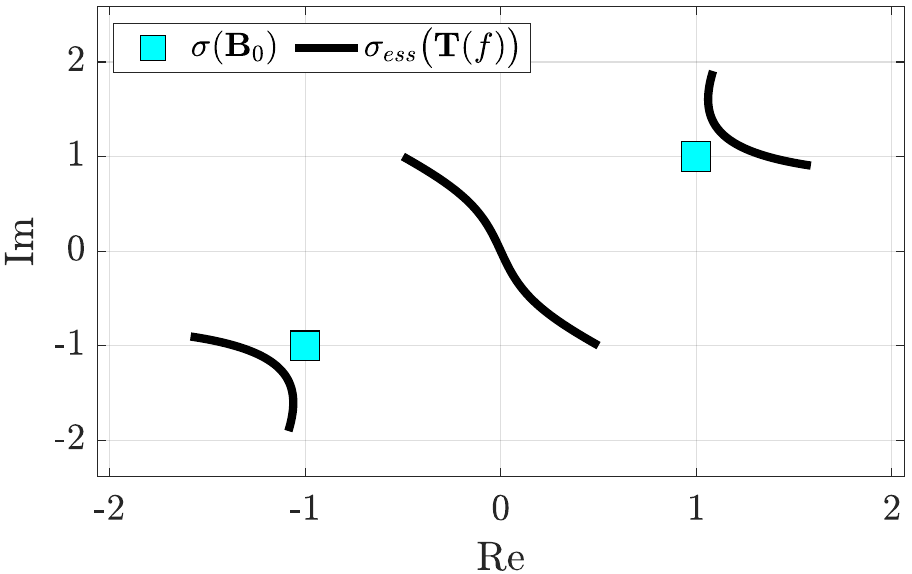}}\quad
    \subfloat[][Computation performed for $t=0$.]%
    {\includegraphics[width=0.45\linewidth]{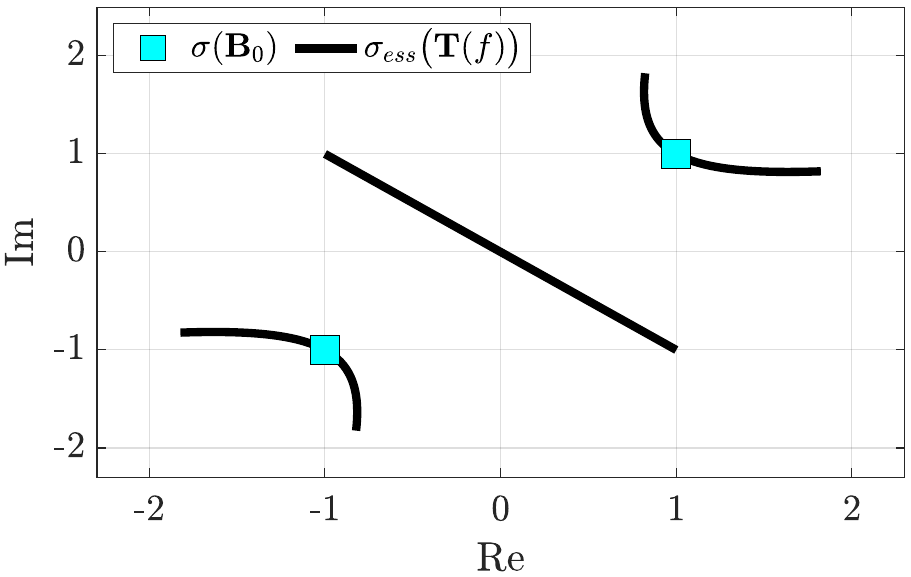}}
    \caption{Computation performed for $a_1 = a_2 = a_3 = 0, b_1 = b_2 = 1+\i$ and $b_3 = e^{\i \pi/2}b_2 + t$.
    Contrary to the real symmetric case as asserted by Theorem \ref{thm: inversion symmetry topological}, in the complex symmetric case, the edge modes $\sigma(\mathbf{B}_0)$ might move into the essential spectrum of $\mathbf{T}(f)$ without the spectral gap closing, despite the symbol function being persymmetric for any $t\in[0,1]$. In addition, the edge modes no longer collides with the essential spectrum at extremities of the bands.}
    \label{Fig: Trimer Gap closing}
\end{figure}

In the complex case, edge modes are topologically protected depending on how the spectral gap is closed.  
\begin{proposition}\label{prop: top protected k geq 3}
    Let $f(z) \in \C^{k\times k}$ for $k\geq3$ be persymmetric, and let $\lambda \in \sigma(\mathbf{B}_0)$ be an edge mode such that $\mathbf{B}_0\mathbf{v} = \lambda \mathbf{v}$ such that as the edge mode approaches the essential spectrum
    \begin{equation}\label{eq: assumption edge mode Floquet}
        \frac{b_k\mathbf{v}_1}{\mathbf{v}_{k-1}b_1} \to \pm 1,
    \end{equation}
    then the edge mode is topologically protected.
\end{proposition}
\begin{proof}
    Let $\lambda \in \sigma(\mathbf{B}_0)$, then by Theorem \ref{cor: coburn eigenvector} it follows that an eigenvector $\mathbf{v}$ of $f(z)$ has to satisfy in the last row the equation
    \begin{equation}
        \mathbf{v_1}b_kz^{-1} + \mathbf{v}_{k-1}b_{k-1} = 0.
    \end{equation}
    From this it follows that the Bloch parameter $z$ is given by 
    \begin{equation}
        z = - \frac{b_k\mathbf{v}_1}{\mathbf{v}_{k-1}b_{k-1}}.
    \end{equation}
    Since $f(z)$ is persymmetric, $f(z)$ is centrosymmetric if and only if $z = \pm 1$. By assumption \eqref{eq: assumption edge mode Floquet}, any continuous transformation of the symbol function that preserves persymmetry, the Floquet parameter $z\to \pm 1$. By the same argument as in the proof of Theorem \ref{thm: inversion symmetry topological}, the edge mode therefore only vanishes into the essential spectrum if the band gaps close, which completes the proof.
\end{proof}

Therefore, adding the additional requirement \eqref{eq: assumption edge mode Floquet} in Proposition \ref{prop: top protected k geq 3} would rule out counterexamples of the kind illustrated in Figure \ref{Fig: Trimer Gap closing}.

\subsection{Asymptotic spectra of Toeplitz matrices}\label{sec: open limit}
We will conclude our study of spectra by presenting a known result on the asymptotic spectra of finite but large tridiagonal $k$-Toeplitz matrices. To define a finite Toeplitz matrix from the semi-infinite Toeplitz operator, we introduce the projection
\begin{align}
    \vect{P}_n: \ell^2(\N, \mathbb{C}) &\to \ell^2(\N, \mathbb{C})\\
    (x_1, x_2, x_3, \dots) &\mapsto (x_1, \dots, x_n, 0, 0, \dots).
\end{align}
A tridiagonal $k$-Toeplitz matrix is derived from a tridiagonal $k$-Toeplitz operator by performing the following truncation,
\begin{equation}
    \vect{T}_{mk}(f) := \vect{P}_{mk} \vect{T}(f) \vect{P}_{mk}.
\end{equation}

We will venture onwards to characterise the \emph{open limit} of complex tridiagonal $k$-Toeplitz matrices. By the open limit, we understand a finite structure which is growing in size and we write
\begin{equation}
    \sigma_{\mathrm{open}}\big(\mathbf{T}(f)\big) :=  \lim_{n\to\infty} \sigma\big(\mathbf{T_n}(f)\big).
\end{equation}

The open limit for Block Toeplitz operators was first established by Widom \cite{WIDOM1974284}. Tridiagonal $k$-Toeplitz operators, being a particular case of block Toeplitz operators, satisfy the assumptions of Widom's theorem.

\begin{theorem}\label{thm: open limit k Toeplitz}
    Let $f \in \C^{k\times k}$ and let $\mathbf{T}_n(f)$ be a tridiagonal $k$-Toeplitz matrix, then it holds that
    \begin{equation}
        \liminf_{n\to\infty} \sigma\big(\mathbf{T_n}(f)\big) = \limsup_{n\to\infty} \sigma\big(\mathbf{T_n}(f)\big) = \lim_{n\to\infty} \sigma\big(\mathbf{T_n}(f)\big) = \Gamma \cup G_0,
    \end{equation}
    where $G_0$ is a set of finite cardinality defined in \eqref{def: G0} and the set $\Gamma$ is a collection of algebraic curves given by \eqref{def: Gamma set}.
\end{theorem}

Note that the sets $\Gamma$ and $G_0$ can be very efficiently computed in the tridiagonal $k$-Toeplitz case using Theorem \ref{thm: edge mode characcterisation} and Lemma \ref{lemma: open limit essential spectrum} respectively.

In particular, since tridiagonal $k$-Toeplitz matrices $\mathbf{T}_n(f) = \mathbf{D}\mathbf{T_n}(f_s)\mathbf{D}^{-1}$ are always similar to symmetric matrices $\mathbf{T}_n(f_s)$, it follows from Lemma \ref{lemma: open limit and essential spectrum} that asymptotically, up to possibly $2k-2$ discrete eigenvalues entailed in $G_0$, the eigenvalues cluster on the essential spectrum of $\mathbf{T}(f_s)$. The open limit presented in Theorem \ref{thm: open limit k Toeplitz} has the advantage over the open limit characterisation in \cite{ammari2024generalisedbrillouinzonenonreciprocal}, that it does not rely on the convergence of pseudospectra and therefore also applies to highly non-normal matrices, and matrices that exhibit exceptional points, as shown in \Cref{Fig: Complex Dimer Gap closing} (B).

\section{Spectra of tridiagonal k-Toeplitz Interface  operators}\label{sec: Twofold Toeplitz}
This section focuses on the spectral analysis of interface Toeplitz operators. We consider two tridiagonal $k$-Toeplitz operators, $\mathbf{T}_A$ and $\mathbf{T}_B$, and assume that they are joined at a common interface, so as to form an \emph{interface Toeplitz operator} of the following form
\begin{equation}\label{def: twofold Toeplitz operator}
    \mathbf{T}_{AB} := \scalebox{0.9}{$
        \begin{pNiceMatrix}[columns-width=0em]
        \Block[draw,fill=blue!31,rounded-corners]{3-3}{} & & & &  \\
        & \mathbf{T}_B& & &  \\
        & &  & q &  \\
        & & q &  \eta & s & & \\
        & & & s & \Block[draw,fill=red!31,rounded-corners]{3-3}{} & &  \\
        & & & & & \mathbf{T}_A &  \\
        & & & & & &  
        \end{pNiceMatrix} $}= \scalebox{0.7}{$
    \begin{pNiceMatrix}[columns-width=0pt]
        \Block[draw,fill=blue!31,rounded-corners]{4-4}{} \scalebox{0.6}{$\scriptstyle\ddots$} & \scalebox{0.6}{$\scriptstyle\ddots$} & {} & {} & {} \\
        \scalebox{0.6}{$\scriptstyle\ddots$} & a'_3 & b'_2 & {} & {} \\
        {} & b'_2 & a'_2 & b'_1 & {} \\
        {} & {} & b'_1 & a'_1 & q \\
        {} & {} & {} & q & \eta & s & {} & {} \\
        {} & {} & {} & {} & s & \Block[draw,fill=red!31,rounded-corners]{4-4}{} a_1 & b_1 & {} \\
        {} & {} & {} & {} & {} & b_1 & a_2 & b_2 \\
        {} & {} & {} & {} & {} & {} & b_2 & a_3 & \scalebox{0.6}{$\scriptstyle\ddots$} \\
        {} & {} & {} & {} & {} & {} & {} & \scalebox{0.6}{$\scriptstyle\ddots$} & \scalebox{0.6}{$\scriptstyle\ddots$}
    \end{pNiceMatrix}
$},
\end{equation}
for some coupling constants $\eta, s,  q \in \C$. We will start by giving a description of the essential spectrum of an interface Toeplitz operator as in \eqref{def: twofold Toeplitz operator}. The $k$-periodic structure of such operators is always specified relative to a canonically selected unit cell. The following theorem shows that the essential spectrum does not depend on which unit cell convention is chosen. We say that $\tilde{\mathbf{T}}_A$ is a \emph{cyclic shift} of $\mathbf{T}_A$ if the $k$-periodic entries of the Toeplitz operator are cyclically shifted, that is, $a_i \to a_{i+m}$, for some $m\in \Z$.
\begin{lemma}\label{lemma: spectrum unit cell convention}
    Let $\mathbf{T}_A$ be a tridiagonal $k$-Toeplitz operator and let $\tilde{\mathbf{T}}_A$ be a cyclic shift of $\mathbf{T}_A$, then
    \begin{equation}
        \sigma_{\mathrm{ess}}(\tilde{\mathbf{T}}_A) = \sigma_{\mathrm{ess}}(\mathbf{T}_A).
    \end{equation}
\end{lemma}
\begin{proof}
    By \cite[Theorem 2.4.]{ammari2024spectra}, the essential spectrum of a $k$-Toeplitz operator is given by
    \begin{equation}
        \sigma_{\mathrm{ess}}\bigl(\mathbf{T}(f)\bigr) = \sigma_{\mathrm{det}}(f).
    \end{equation}
    The assertion then follows by the fact that the determinant is invariant under cyclic shifts.
\end{proof}
The essential spectrum of interface Toeplitz operators admits the following description.
\begin{theorem}\label{theorem: essential spectrum for twofold}
 Let $\mathbf{T}_{AB}$ be the interface Toeplitz operator defined in \eqref{def: twofold Toeplitz operator}, then the essential spectrum is given by
    \begin{equation}
        \sigma_{\mathrm{ess}}\bigl(\mathbf{T}_{AB}\bigr) = \sigma_{\mathrm{ess}}\bigl(\mathbf{T}_{A}\bigr) \cup \sigma_{\mathrm{ess}}\bigl(\mathbf{T}_{B}\bigr).
    \end{equation}
\end{theorem}

\begin{proof}
    The operator $\mathbf{T}_{AB}$ defined in \eqref{def: twofold Toeplitz operator} is a compact perturbation of the block diagonal operator
    \begin{equation}
        \mathbf{T}_D = \begin{pmatrix}
            \mathbf{T}_{B} & 0 \\
            0 & \tilde{\mathbf{T}}_{A}
        \end{pmatrix} = \mathbf{T}_{B}\oplus \tilde{\mathbf{T}}_{A}.
    \end{equation}
    By the direct sum decomposition of the essential spectrum, it holds
    \begin{equation}
        \sigma_{\mathrm{ess}}\bigl(\mathbf{T}_B \oplus \tilde{\mathbf{T}}_A \bigr) = \sigma_{\mathrm{ess}}\bigl(\mathbf{T}_B \bigr) \cup \sigma_{\mathrm{ess}}\bigl( \tilde{\mathbf{T}}_A \bigr),
    \end{equation}
    and since the essential spectrum is unchanged under compact perturbations it holds that
    \begin{equation}
        \sigma_{\mathrm{ess}}\bigl(\mathbf{T}_{AB}\bigr) = \sigma_{\mathrm{ess}}\bigl(\mathbf{T}_D\bigr) = \sigma_{\mathrm{ess}}\bigl(\mathbf{T}_B \bigr) \cup \sigma_{\mathrm{ess}}\bigl( \tilde{\mathbf{T}}_A \bigr) = \sigma_{\mathrm{ess}}\bigl(\mathbf{T}_B \bigr) \cup \sigma_{\mathrm{ess}}\bigl( \mathbf{T}_A \bigr) ,
    \end{equation}
    where in the last step, we used the invariance of the essential spectrum under cyclic shifts as asserted by Lemma \ref{lemma: spectrum unit cell convention}, which completes the proof.
\end{proof}
The discrete spectrum of tridiagonal interface Toeplitz operators is more subtle to characterise. In the sequel, we therefore exploit the symmetries in the structure of the building blocks $\mathbf{T}_A$ and $\mathbf{T}_B$.

\subsection{Reflection-symmetric operators.}\label{sec: relfection symmetry}
In this section, we focus on interface Toeplitz operators under the additional assumption that their building blocks, $\mathbf{T}_A$ and $\mathbf{T}_B$, are reflection- (or mirror-) symmetric.
This is of particular interest, as it places the system in the same form as the celebrated SSH model~\cite{PhysRevLett.42.1698} under global inversion symmetry.
\begin{definition}[Interface eigenvalue]
We say that $\lambda\in\C$ is an \emph{interface eigenvalue} if it lies in the common spectral gap, that is $\lambda \in \sigma(\mathbf{T}_{AB}) \setminus \sigma_{\mathrm{ess}}(\mathbf{T}_{AB})$ and if the associated \emph{interface eigenmode} is decaying away from the interface into both bulks.
\end{definition}
Let us introduce the reflection operator,
\begin{align}\label{def: refelction operator}
    \mathbf{R}: \ell^2(\Z_{\geq 0}) &\to \ell^2(\Z_{\leq 0})\\
    (\vect{v}_0, \vect{v}_1, \dots) &\mapsto (..., \vect{v}_1, \vect{v}_0). \nonumber
\end{align}
Let $\mathbf{T}_A$ be a tridiagonal $k$-Toeplitz operator, then its mirrored counterpart is given by $\mathbf{T}_B = \mathbf{R} \mathbf{T}_A \mathbf{R}$ and we have $\mathbf{R}^2 = \Id$. The following result establishes that the eigenvectors inherit the same mirror symmetry.
\begin{lemma}\label{lemma: eigenvector reflection}
    Let $\lambda$ be an eigenvalue of $\mathbf{T}_A$ with eigenvector $\vect{v}$, then $\lambda$ is also an eigenvalue for $\mathbf{T}_B$ with eigenvector $\mathbf{R}\vect{v}$.
\end{lemma}
\begin{proof}
    A direct computation yields
    \begin{equation}
        \mathbf{T}_B \bigl(\mathbf{R}\vect{v}\bigr) = \mathbf{R} \mathbf{T}_A \mathbf{R} \bigl(\mathbf{R}\vect{v}\bigr) = \mathbf{R}\bigl(\mathbf{T}_A \vect{v}\bigr) = \mathbf{R}\bigl(\lambda \vect{v}\bigr) = \lambda \bigl(\mathbf{R}\vect{v}\bigr),
    \end{equation}
    which completes the proof.
\end{proof}
This symmetry is inherited by the eigenvectors of the interface operator. We have the following categories.

\begin{definition}[Monopole and Dipole modes]\label{def: monopole dipole}
    We say that an eigenmode is a \emph{monopole mode} if it has even symmetry
    \begin{equation}
        \vect{w}_i = \vect{w}_{-i}.
    \end{equation}
    A \emph{dipole mode} , on the other hand, has odd symmetry, that is
    \begin{equation}
        \vect{w}_i = -\vect{w}_{-i}.
    \end{equation}
\end{definition}

We have the following result on the parity of the eigenmodes of a reflection-symmetric operator $\mathbf{T}_{AB}$.
\begin{lemma}\label{lemma: parity}
    Suppose that $\mathbf{T}_{AB}$ satisfies \eqref{eq: mirror symmetry}, then any eigenmode $\vect{w}$ corresponding to a simple eigenvalue $\lambda$ is either an monopole or dipole mode.
\end{lemma}
\begin{proof}
        Let $\mathbf{R}$ be the reflection operator defined in \eqref{def: refelction operator}, then by the symmetry assumption on $\mathbf{T}_{AB}$ it holds that $\mathbf{R}\mathbf{T}_{AB} = \mathbf{T}_{AB} \mathbf{R}$. Let $\lambda \in \sigma\bigl(\mathbf{T}_{AB}\bigr) \setminus \sigma_{\mathrm{ess}}\bigl(\mathbf{T}_{AB} \bigr)$, then $\mathbf{T}_{AB}\vect{v} = \lambda \vect{v}$. 
        Applying $\mathbf{R}$ to both sides yields,
        \begin{equation}
            \mathbf{T}_{AB}\bigl(\mathbf{R}\vect{v}\bigr) = \mathbf{R}\bigl(\mathbf{T}_{AB}\vect{v}\bigr) = \mathbf{R}(\lambda \vect{v}) = \lambda \bigl(\mathbf{R}\vect{v}\bigr).
        \end{equation}
 Thus, $\mathbf{R}\vect{v}$ is also an eigenvector of $\mathbf{T}_{AB}$ with the same eigenvalue. Since $\lambda$ is simple, it must hold that the two eigenvectors are linearly dependent, i.e $\mathbf{R}\vect{v} = c\vect{v}$ for some $c \in \R$. Since  $\mathbf{R}^2 = \Id$ it follows that $\mathbf{R}\mathbf{R}\vect{v} = c^2 \vect{v} = v$, so that $c^2 = 1$, which implies $ c = \pm 1$.
\end{proof}

Let us now examine the interface Toeplitz operator from \eqref{def: twofold Toeplitz operator}, which is constructed from two reflection-symmetric Toeplitz operators. It is easy to verify that this type of operator, which is mirror-symmetric with respect to the interface, satisfies the relation
\begin{equation}\label{eq: mirror symmetry}
    \bigl(\mathbf{T}_{AB}\bigr)_{i,j} = \bigl(\mathbf{T}_{AB}\bigr)_{-i,-j}.
\end{equation}In that case, it follows directly from Theorem \ref{theorem: essential spectrum for twofold} that the essential spectrum is given by,
\begin{equation}
    \sigma_{\mathrm{ess}}\bigl(\mathbf{T}_{AB}\bigr) = \sigma_{\mathrm{ess}}\bigl(\mathbf{T}_A\bigr).
\end{equation}
Characterising the point spectrum of $\mathbf{T}_{AB}$ is more subtle, as it can originate from different mechanisms, such as edge-induced interfaces or eigenvalue-matched interface modes. 

\subsection{Localised modes for interface Toeplitz operators}
In this section, we demonstrate that interface Toeplitz operators admit two distinct mechanisms for supporting localised interface modes. The first arises naturally from composition. If a Toeplitz matrix $\mathbf{T}_A$ supports edge modes, then the composite interface Toeplitz matrix $\mathbf{T}_{AB}$ will automatically inherit an interface mode. The second mechanism is more intricate, which arises when an eigenvalue is tuned to match a prescribed eigenvector structure.
We proceed to discuss both scenarios.

\begin{theorem}\label{thm: spectral inclusion}
    Let $\mathbf{T}_{AB}$ be reflection-symmetric as in \eqref{eq: mirror symmetry} and let $\mathbf{T}_A$ be as in \eqref{def: twofold Toeplitz operator}, then the spectral inclusion 
    \begin{equation}
        \sigma\bigl(\mathbf{T}_A\bigr) \subseteq \sigma\bigl(\mathbf{T}_{AB}\bigr)
    \end{equation}
    holds. Moreover, if $\vect{v}$ is an eigenvector of $\mathbf{T}_A$, such that $\mathbf{T}_A\vect{v} = \lambda \vect{v}$ and $\mathbf{v}_1\neq 0$ , then there exists an eigenvector $\mathbf{w}$, such that $\mathbf{T}_{BA}\vect{w} = \lambda \vect{w}$, where the eigenvector is explicitly given by
    \begin{equation}\label{eq: eqigenvector twwofold Toeplitz}
        \vect{w}_i = \begin{cases}
            -\vect{v}_{|i|}, & i < 0,\\
            0, &i = 0,\\
            \vect{v}_i, & i > 0.
        \end{cases}
    \end{equation}
\end{theorem}

\begin{proof}
    By Lemma \ref{lemma: eigenvector reflection} we have $\mathbf{v}_B = \mathbf{R}\mathbf{v}_A$ and let us make the following eigenvector Ansatz,
     
    \begin{equation}\label{eq:structuredtwofoldToeplitz} 
    \scalebox{0.8}{$
         \left[
        \begin{pNiceMatrix}[columns-width=0em]
        \Block[draw,fill=blue!31,rounded-corners]{3-3}{} & & & &  \\
        & \mathbf{T}_B& & &  \\
        & &  & q &  \\
        & & q &  \eta & q & & \\
        & & & q & \Block[draw,fill=red!31,rounded-corners]{3-3}{} & &  \\
        & & & & & \mathbf{T}_A &  \\
        & & & & & &  
        \end{pNiceMatrix}  - \lambda \mathrm{Id} \right] $}
        \scalebox{0.7}{$
        \begin{pNiceMatrix}[columns-width=0em]
        a{\color{blue!80}\begin{pmatrix}
            \vdots \\
            \vect{v}_2 \\
            \vect{v}_1
        \end{pmatrix}}\\
        t\\
        {\color{red!80}\begin{pmatrix}
            \vect{v}_1 \\
            \vect{v}_2\\
            \vdots
        \end{pmatrix}}
        \end{pNiceMatrix}$} = 
        \scalebox{0.7}{$\begin{pmatrix}
            \vdots \\
            0 \\
            qt \\
            aq\vect{v}_1 + (\eta-\lambda)t + q\vect{v}_1\\
            qt \\
            0 \\
            \vdots
        \end{pmatrix}$}.
    \end{equation}
    
 Consequently, it must hold that $t = 0$, then the residual becomes $\vect{v}_1(1+a)q$, which vanishes if $a = -1$, which proves the structure of the eigenvector $\vect{w}$ given in \eqref{eq: eqigenvector twwofold Toeplitz}.
\end{proof}

Observe that the above proof applies equally to finite matrices and infinite operators. In particular, to obtain a localised interface mode for the  tridiagonal interface Toeplitz matrix, it suffices to identify a Toeplitz matrix $\mathbf{T}_A$ that supports an edge mode.

In certain situations, the interface Toeplitz operator possesses eigenvalues within the bandgap that do not belong to $\sigma(\mathbf{B}_0)$. Such eigenvalues must satisfy a discrete counterpart of an impedance-matching condition at the interface, as characterised by the following result.

\begin{theorem}\label{thm: refexion symmetric impedance matching}
    Suppose $\mathbf{T}_{AB}$ is as in \eqref{eq: mirror symmetry}, let $f$ be the symbol function of $\mathbf{T}_A$ and let $\mathbf{v}$, be such that
    \begin{equation}
        \Big[f\big(z_1(\lambda)\big)-\lambda\Id\Big]\vect{v} = 0
    \end{equation}
    and assume that $\mathbf{v}_k \neq 0$, $\lambda\not\in\sigma_{\mathrm{ess}}(\mathbf{T}_{AB})$ and let $b_1$ be the first lower diagonal entry of $\mathbf{T}_A$. Then if $\lambda$ satisfies,
    \begin{equation}\label{eq: interface frequency}
        \lambda = \frac{2q^2z_1(\lambda)\mathbf{v}_1}{b_1\mathbf{v}_k} + \eta,
    \end{equation} 
    $\lambda$ is an interface eigenvalue of $\mathbf{T}_{AB}$, in which case $\mathbf{T}_{AB}\vect{w} = \lambda \vect{w}$, where
    \begin{equation}\label{eq: impedance matched interface}
        \mathbf{w}_i = \begin{cases}
            \left(\mathbf{R}\vect{u}\right)_{|i|}, & i < 0 \\
            z_1^{-1}\vect{v}_k, & i = 0\\
            \vect{u}_i, & i > 0
        \end{cases}
    \end{equation}
    where $\vect{u} = \big(z_1^0\vect{v}, z_1^1\vect{v}, z_1^2\vect{v}, \dots \big)$.
\end{theorem}

\begin{proof}
    Let $\mathbf{u}$, be an eigenvector of the doubly infinite Laurent operator to the eigenvalue $\lambda$, such that $\mathbf{L}_{A}\mathbf{u} = \lambda \mathbf{u}$. Then the eigenvector is constructed as follows,
    \begin{equation}
        \mathbf{u} = \big(\dots, z_1^{-2}\vect{v}, z_1^{-1}\vect{v}, z_1^0\vect{v}, z_1^1\vect{v}, z_1^2\vect{v}, \dots \big)
    \end{equation}
    where $|z_1|< 1$ is the root of the smallest magnitude of $\operatorname{det}\big(f(z)-\lambda\Id\big) = 0$, and the Bloch vector $\mathbf{v}$ satisfies $\big[f(z_1)-\lambda\Id\big]\vect{v} = 0$. Moreover, fixing the origin in the indexing of $\mathbf{u}$ as $\mathbf{u}_1 = z_1^0\vect{v}_1$ and $\mathbf{u}_{-1} = z_1^{-1}\vect{v}_k$. 
    Then the following eigenvalue problem becomes
    \begin{equation}\label{eq: construction eigenproblem twofold}
        \scalebox{0.8}{$ \left[
        \begin{pNiceMatrix}[columns-width=0em]
        \Block[draw,fill=blue!31,rounded-corners]{3-3}{} & & & &  \\
        & \mathbf{T}_B& & &  \\
        & &  & q &  \\
        & & q &  \eta & q & & \\
        & & & q & \Block[draw,fill=red!31,rounded-corners]{3-3}{} & &  \\
        & & & & & \mathbf{T}_A &  \\
        & & & & & &  
        \end{pNiceMatrix} - \lambda \mathrm{Id} \right]$} \scalebox{0.7}{$
        \begin{pNiceMatrix}[columns-width=0em]
        a{\color{blue!80}\begin{pmatrix}
            \vdots \\
            \vect{u}_2 \\
            \vect{u}_1
        \end{pmatrix}}\\
        t\\
        {\color{red!80}\begin{pmatrix}
            \vect{u}_1 \\
            \vect{u}_2\\
            \vdots
        \end{pmatrix}}
        \end{pNiceMatrix}$} = 
        \scalebox{0.7}{$
        \begin{pmatrix}
            \vdots \\
            0 \\
            qt - a(\mathbf{L}_B)_{0,1}\vect{u}_{-1} \\
            aq\vect{u}_1 + (\eta-\lambda)t + q\vect{u}_1\\
            qt - (\mathbf{L}_A)_{0,-1}\vect{u}_{-1} \\
            0 \\
            \vdots
        \end{pmatrix}$}.
    \end{equation}
    From the symmetry \eqref{eq: mirror symmetry} it follows $(\mathbf{L}_B)_{0,1} = (\mathbf{L}_A)_{0,-1} = b_1$, where the $b_1$ is the same as from the Toeplitz structure.
    First, from $qt - ab_1\vect{u}_{-1} = qt - b_1\vect{u}_{-1} = 0$ it follows that $a = 1$. Note that $\vect{u}_{-1} = z_1^{-1}\mathbf{v}_k$, from which we deduce that $t = z_1^{-1}\mathbf{v}_kb_1/q$. Substituting into $aq\vect{u}_1 + (\eta-\lambda)t + q\vect{u}_1 = 0$ and solving for $\lambda$ yields,
    \begin{equation}\label{eq: lambda impedance}
        \lambda = \frac{2q^2z_1(\lambda)\mathbf{v}_1}{b_1\mathbf{v}_k} + \eta.
    \end{equation}
    The structure of the eigenvector is an immediate consequence of the construction presented in \eqref{eq: construction eigenproblem twofold}. Since the root of the smallest magnitude was chosen, the vector $\mathbf{u}$ is exponentially decaying in the index, so the interface eigenvector $\mathbf{w}$ in \eqref{eq: impedance matched interface} is exponentially localised around the interface, which completes the proof.
\end{proof}

A direct observation from the eigenvector construction in Theorem \ref{thm: spectral inclusion} is that edge induced interface modes are dipole modes. On the other hand, interface modes supported with eigenvalues that satisfy \eqref{eq: interface frequency} as presented in Theorem \ref{thm: refexion symmetric impedance matching} are monopole modes. This will be further investigated and numerically illustrated in a physical setting with damped resonator chains in Section \ref{Sec: resonator chains}.

Another important observation is that interface modes cannot be both edge-induced and interface-matched.
This is because for $\lambda \in \sigma(\mathbf{B}_0)$ by Theorem \ref{cor: coburn eigenvector} it follows $\mathbf{v}_k = 0$, which would violate an important assumption in Theorem \ref{thm: refexion symmetric impedance matching} and vice versa. Both types of interface eigenvalues will be numerically illustrated in Figure \ref{Fig: Interface Twofold Dimer}.

\begin{figure}[htb]
    \centering
    \subfloat[][Computation performed for $a_1 = 1.2, a_2 = 2, b_1 =c_1 =0.8-0.4\i$ and $b_2 = c_2 = 1.2-0.2\i$.]%
    {\includegraphics[width=0.453\linewidth]{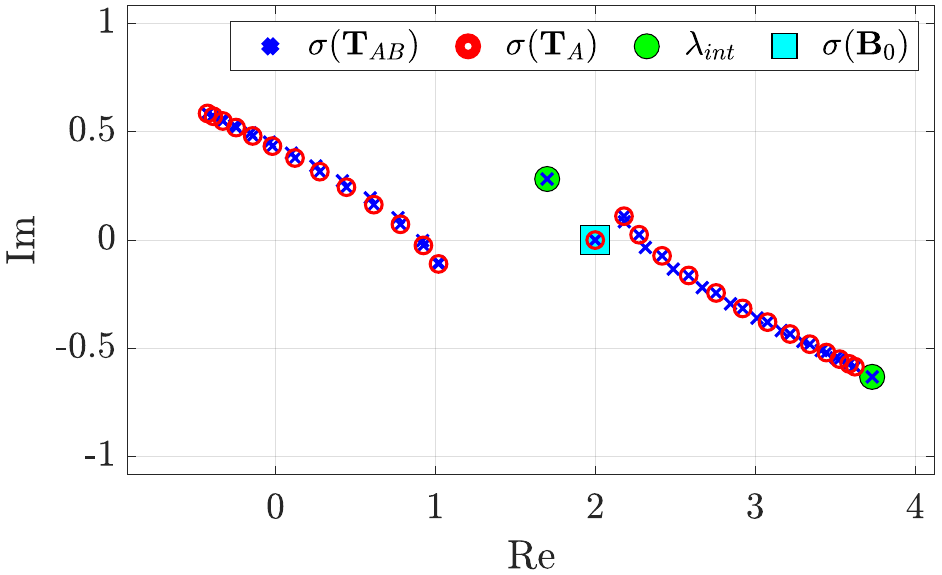}}\quad
    \subfloat[][Computation performed for $a_1 = 1.2,~ a_2 = 1,~ a_3 = 1.4 + 0.1\i,~  b_1 =c_1 =3.8+0.4\i$, $b_2 = c_2 = 1.2-0.9\i$ and $b_3 = c_3 = 2 + 0.1\i$.]%
    {\includegraphics[width=0.45\linewidth]{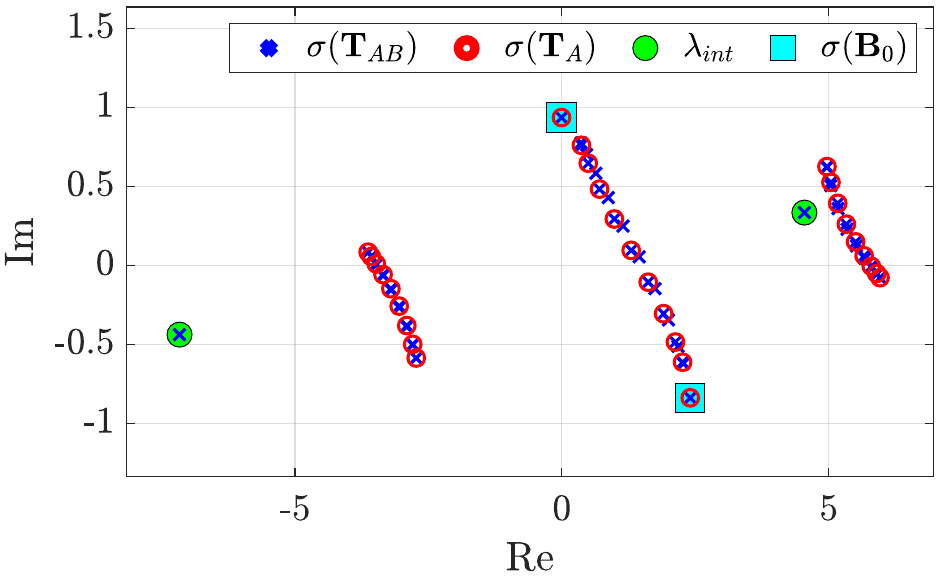}}
    \caption{Symmetric interface Toeplitz operators with complex valued entries. The value $\lambda_{\text{int}}$ is a numerically computed solution to the fixed-point problem in \eqref{eq: interface frequency} and correctly predicts all admissible interface eigenvalues, while $\sigma(\mathbf{B}_0)$ correctly predicts all edge induced interface eigenvalues.}
    \label{Fig: Interface Twofold Dimer}
\end{figure}
Equation~\eqref{eq: lambda impedance} naturally defines the function
\begin{equation}\label{eq: impedance analogue}
    F(\lambda) = \frac{2q^2z_1(\lambda)\mathbf{v}_1}{b_1\mathbf{v}_k} - \lambda + \eta,
\end{equation}
whose roots $\lambda \in \mathbb{C}$ are precisely the interface-matched eigenvalues. In addition, the function $F(\lambda)$ lends itself to efficient numerical root-finding, which is illustrated in Figure \ref{Fig: Impedance matched}.

\begin{figure}[htb]
    \centering
    \subfloat[][Computation performed for $a_1 = 1.2, a_2 = 2, b_1 =c_1 =0.8-0.4\i$ and $b_2 = c_2 = 1.2-0.2\i$.]%
    {\vspace{0.7mm}\includegraphics[width=0.465\linewidth]{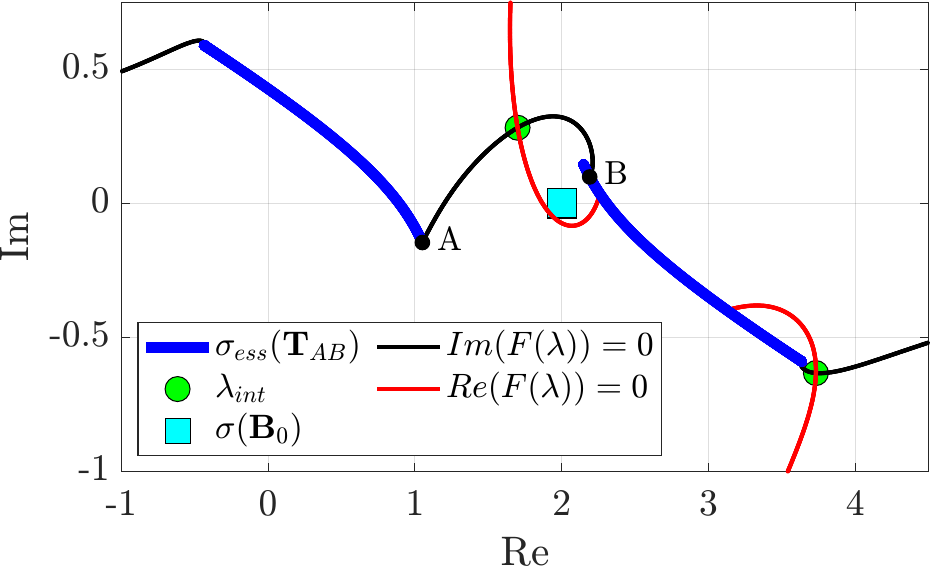}}\quad
    \subfloat[][The function $\mathrm{Re}\big(F(\lambda)\big)$, evaluated on $\mathrm{Im}\big(F(\lambda)\big)=0$ can be seen as an analogue to the discrete impedance in discrete damped systems.]%
    {\includegraphics[width=0.45\linewidth]{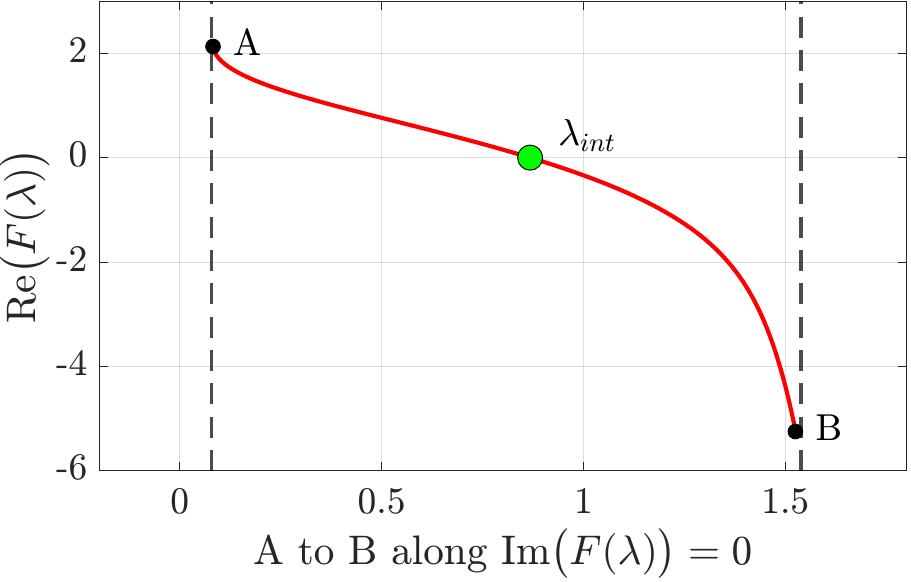}}
    \caption{The function $F(\lambda)$ may be seen as a discrete analogue of the impedance. Since the values close to the essential spectrum, that is $A$ and $B$ in Figure (B) stay finite, the interface mode is no longer topologically protected. Indeed adding a perturbation, which increases $\eta$ enough so that there does no longer exists a root to $F(\lambda) = 0$.  Computation performed for the same numerical values as in Figure \ref{Fig: Interface Twofold Dimer} (A).}
    \label{Fig: Impedance matched}
\end{figure}

The interface Toeplitz operator may also be constructed by  directly coupling the building blocks $\mathbf{T}_A$ and $\mathbf{T}_B$ by a common coupling without sharing a common interface site. As the results are similar to this section, we will not present them here and refer the reader to the Appendix \ref{Sec: common coupling}.

We conclude this section by examining the distinct nature of topological protection for interface modes. For the edge-induced interface modes of $\mathbf{T}_{AB}$ described in Theorem~\ref{thm: spectral inclusion}, the topological protection follows from that of the underlying edge modes, provided certain mild symmetry conditions on the symbol function are satisfied, as stated in Theorem~\ref{thm: inversion symmetry topological}, Propositions~\ref{prop: complex dimer topology} and~\ref{prop: top protected k geq 3}. In contrast, there is no analogous result on the topological protection of interface modes induced by a matched interface eigenvalue as given by Theorem~\ref{thm: refexion symmetric impedance matching}. Consequently, these modes lack topological protection and may disappear under compact perturbations of the coupling parameters $q$ and $\eta$ as already pointed out in \cite[Section 4]{10.1098/rspa.2022.0675}. These results are illustrated numerically for damped resonator chains in the following section, particularly in Figure \ref{Fig: Not topological}.

\section{Damped one dimensional resonator chains.}\label{Sec: resonator chains}
In one-dimensional resonator chains, the interaction is limited to nearest neighbours. Consequently, in the corresponding mathematical description, either via a tight-binding Hamiltonian or, as in this work, through a capacitance matrix, the resulting matrices 
are tridiagonal. Moreover, the periodicity in the resonator arrangements of the underlying structure is reflected in 
these matrices. Thus, in one dimension, analysing resonances in $k$-periodic resonator chains is equivalent 
to studying tridiagonal $k$-Toeplitz operators.
In this work, we consider the one-dimensional differential eigenvalue problem with damping,
\begin{equation}\label{eq: differential operator}
    \omega^2u = -\frac{1}{\mu_0}\frac{\partial}{\partial x} \left( \frac{1}{\varepsilon(x)} \frac{\partial u}{\partial x}\right).
\end{equation}
The magnetic permeability $\mu_0\in \R_{>0}$ is assumed to be constant. The permittivity $\varepsilon \colon \R \to \C$ whose imaginary part encodes damping is defined piecewise on the resonator domain $D := \bigcup_{i=1}^N D_i$ as
\begin{equation}\label{eq: piecewise permittivity}
    \varepsilon(x) = \begin{cases}
        \varepsilon_b, \quad &x \in D, \\
        \varepsilon, &x \in \R \setminus D.
    \end{cases}
\end{equation}
The wave speeds inside the resonators $D$ and in the background medium $\R \setminus D$ are given, respectively, by
\begin{equation}\label{eq: wavedpeed definition}
    v_b := \frac{1}{\sqrt{\mu_0\varepsilon_b}}, \qquad v := \frac{1}{\sqrt{\mu_0\varepsilon}},
\end{equation}
where $\delta := \varepsilon_b / \varepsilon$ denotes the contrast parameter. The wave speeds inside the resonators throughout a resonator chain comprising $N$ resonators will be saved in the matrix,
\begin{equation}\label{eq: wave speeds}
    V = \operatorname{diag}(v_b, \dots, v_b) \in \C^{N\times N}.
\end{equation}
For the stepwise defined materials, the differential eigenvalue problem \eqref{eq: differential operator} reduces to the boundary problem 
\begin{align}
\scalebox{0.9}{$
    \label{eq: system of coupled equations}
    \begin{dcases}
        \frac{\dd{^2}}{\dd x^2}u(x)+ \frac{\omega^2}{v^2}u(x) = 0, & x\in \R \setminus D ,\\
        \frac{\dd{^2}}{\dd x^2}u(x)+ \frac{\omega^2}{v_b^2}u(x) = 0, & x\in D ,\\
        u\vert_{\iR}(x^{\iLR}_{{i}}) - u\vert_{\iL}(x^{\iLR}_{{i}}) = 0, &  1\leq i\leq N ,\\
        \left.\frac{\dd u}{\dd x}\right\vert_{\iR}(x^{\iL}_{{i}}) - \delta\left.\frac{\dd u}{\dd x}\right\vert_{\iL}(x^{\iL}_{{i}}) = 0, & 1\leq i\leq N ,\\
        \delta\left.\frac{\dd u}{\dd x}\right\vert_{\iR}(x^{\iR}_{{i}}) - \left.\frac{\dd u}{\dd x}\right\vert_{\iL}(x^{\iR}_{{i}}) = 0, & 1\leq i\leq N,\\
        \big(\frac{\dd}{\dd |x|} - \mathrm{i} k \big) u = 0, & \text{for } x \in (-\infty, x^{\iL}_{1}) \cup (x^{\iR}_{N}, +\infty).  \\
    \end{dcases}$}
\end{align}

A common approach to realise subwavelength interactions is to employ high-contrast metamaterials, which are formed by embedding a collection of strongly contrasted resonators into a background medium such that $0 < \delta \ll 1$. We are interested in non-trivial frequencies $\omega$ such that the scattering problem \eqref{eq: system of coupled equations} is satisfied and such that $\omega \to 0$ as $\delta \to 0$. Such resonances will be denoted \emph{subwavelength resonances}.
It has been demonstrated in \cite{doi:10.1137/22M1503841} that within the subwavelength regime the resonances can be approximated in terms of a discrete capacitance matrix formulation.
The generalised capacitance matrix, taking into account the wave speeds \eqref{eq: wave speeds}, is of the form 
\begin{equation}\label{eq: generalised capacitance}
    \mathcal{C} := V C \in \C^{N\times N}.
\end{equation}
The resonances to \eqref{eq: system of coupled equations} are recovered by the fundamental theorem of capacitance. 
\begin{proposition}\label{prop: reduction to capacitance matrix}
Consider a system of $N$ subwavelength resonators and assume that the eigenvalues of the generalised capacitance matrix $\mathcal{C}$ defined in \eqref{eq: generalised capacitance} are simple. Then, the $N$ subwavelength resonant frequencies $\omega_i$ of \eqref{eq: system of coupled equations} satisfy to the first order
    \begin{align*}
        \omega_i = \sqrt{\delta\lambda_i} + \mathcal{O}(\delta),~i\in\{1, \dots, N\}
    \end{align*}
    where $(\lambda_i)_{1\leq i\leq N}$ are the eigenvalues of the generalised capacitance matrix \eqref{eq: generalised capacitance}.
\end{proposition}

We proceed by examining interface modes in damped subwavelength resonator chains, which arise when two mirrored materials are joined at an interface. Particular attention is paid to distinguishing between topologically protected interface modes (see Definition~\ref{def: topologically protected}) and topologically unprotected interface modes, which may vanish under compact perturbations. A representative example of such a resonator chain is shown in Figure \ref{fig: geometrical defect}.

\begin{figure}[H]
    \centering
    \begin{adjustbox}{width=\textwidth}
        \begin{tikzpicture}
        \draw[|-|,dashed] (0,0.5) -- (1.5,0.5);
        \node[above] at (0.75,0.5) {\LARGE$s_2$};
        \draw[line width=4pt, red] (1.5,0) -- (2.5,0);
        \node[below] at (2,0) {$D_{1}$};
        \draw[-,dotted] (1.5,0) -- (1.5,0.5);
        
        \draw[|-|,dashed] (2.5,0.5) -- (3.0,0.5);
        \node[above] at (2.75,0.5) {\LARGE$s_1$};
        \draw[line width=4pt, red] (3,0) -- (4,0);
        \node[below] at (3.5,0) {$D_{2}$};
        \draw[-,dotted] (2.5,0) -- (2.5,0.5);
        \draw[-,dotted] (3.0,0) -- (3.0,0.5);
        \draw[-,dotted] (4,0) -- (4,0.5);
        
        \begin{scope}[shift={(+4,0)}]
        \draw[|-|,dashed] (0.0,0.5) -- (1.5,0.5);
        \node[above] at (0.75,0.5) {\LARGE$s_2$};
        \draw[line width=4pt, red] (1.5,0) -- (2.5,0);
        \node[below] at (2,0) {$D_{3}$};
        \draw[-,dotted] (1.5,0) -- (1.5,0.5);
        \node at (4.75,.25) {\dots};
        \end{scope}

        \begin{scope}[shift={(+7.5,0)}]
        \draw[line width=4pt, red] (-0.5,0) -- (0.5,0);
        \node[below] at (0.0,0) {$D_{4}$};
        \draw[-,dotted] (-1,0) -- (-1,0.5);
        \draw[-,dotted] (-0.5,0) -- (-0.5,0.5);
        \draw[|-|,dashed] (-1.0,0.5) -- (-0.5,0.5);
        \node[above] at (-0.75,0.5) {\LARGE$s_1$};
        \draw[-,dotted] (2,-0.5) -- (2,1.0);
        \draw[-,dotted] (0.5,-0.5) -- (0.5,1.0);
        
        \draw[|-|,dashed] (2,0.5) -- (3.5,0.5);
        \node[above] at (2.75,0.5) {\LARGE$s_2$};
        \draw[line width=4pt, red] (3.5,0) -- (4.5,0);
        \node[below] at (4,0) {$D_{2m-1}$};
        \draw[-,dotted] (2,0) -- (2,0.5);
        \draw[-,dotted] (3.5,0) -- (3.5,0.5);
        \draw[-,dotted] (4.5,0) -- (4.5,0.5);
\begin{scope}[shift={(+2.5,0)}]
        \draw[|-|,dashed] (2,0.5) -- (2.5,0.5);
        \node[above] at (2.25,0.5) {\LARGE$s_1$};
        \draw[line width=4pt, red] (2.5,0) -- (3.5,0);
        \node[below] at (3,0) {$D_{2m}$};
        \draw[-,dotted] (3.5,-0.5) -- (3.5,1.0);
        \draw[-,dotted] (2.5,0) -- (2.5,0.5);
        \end{scope}
        \end{scope}
        
\begin{scope}[shift={(-4,0)}]

        \draw[|-|,dashed] (0,0.5) -- (0.5,0.5);
        \node[above] at (0.25,0.5) {\LARGE$s_1$};
        \draw[line width=4pt, blue] (0.5,0) -- (1.5,0);
        \node[below] at (1,0) {$D_{-1}$};
        \draw[-,dotted] (0.5,0) -- (0.5, 0.5);
        
        \draw[|-|,dashed] (1.5,0.5) -- (3,0.5);
        \node[above] at (2.25,0.5) {\LARGE$s_2$};
        \draw[line width=4pt] (3,0) -- (4,0);
        \node[below] at (3.5,0) {$D_{0}$};
        \draw[-,dotted] (1.5,0) -- (1.5,0.5);
        \draw[-,dotted] (3,0) -- (3,0.5);
        \draw[-,dotted] (4,0) -- (4,0.5);

        \draw[-,dotted] (-1.0,-0.5) -- (-1.0,1.0);
        \draw[-,dotted] (4.0,-0.5) -- (4.0,1.0);
        \draw[-,dotted] (3.0,-0.5) -- (3.0,1.0);
        \draw[-,dotted] (8.0,-0.5) -- (8.0,1.0);
        \end{scope}
        
        \begin{scope}[shift={(-8.5,0)}]
        \node at (-1.25,.25) {\dots};
        \draw[line width=4pt, blue] (1.0,0) -- (2.0,0);
        \node[below] at (1.5,0) {$D_{-3}$};
        
        \draw[|-|,dashed] (2.0,0.5) -- (3.5,0.5);
        \node[above] at (2.75,0.5) {\LARGE$s_2$};
        \draw[line width=4pt, blue] (3.5,0) -- (4.5,0);
        \node[below] at (4,0) {$D_{-2}$};
        \draw[-,dotted] (2.0,0) -- (2.0,0.5);
        \draw[-,dotted] (3.5,0) -- (3.5,0.5);
        \draw[-,dotted] (4.5,0) -- (4.5,0.5);
        \end{scope}

        \begin{scope}[shift={(-15.5,0)}]
        \draw[line width=4pt, blue] (1,0) -- (2,0);
        \draw[-,dotted] (1,-0.5) -- (1,1.0);
        \node[below] at (1.5,0) {$D_{-2m}$};

        \node[below] at (-1.5,0.5) {\huge Arrangement $A$: };

        \draw[|-|,dashed] (2,0.5) -- (2.5,0.5);
        \node[above] at (2.25,0.5) {\LARGE$s_1$};
        \draw[line width=4pt, blue] (2.5,0) -- (3.5,0);
        \node[below] at (3,0) {$D_{-2m+1}$};
        \draw[-,dotted] (2,0) -- (2,0.5);
        \draw[-,dotted] (2.5,0) -- (2.5,0.5);
        \draw[-,dotted] (3.5,0) -- (3.5,0.5);
\begin{scope}[shift={(+2,0)}]
        \draw[|-|,dashed] (1.5,0.5) -- (3,0.5);
        \node[above] at (2.25,0.5) {\LARGE$s_2$};
        \draw[-,dotted] (3,-0.5) -- (3,1.0);
        
        \draw[line width=4pt, blue] (4.5,0) -- (5.5,0);
        \node[below] at (5,0) {$D_{-4}$};

        \draw[|-|,dashed] (5.5,0.5) -- (6,0.5);
        \draw[-,dotted] (5.5,0) -- (5.5,0.5);
        \draw[-,dotted] (6,0) -- (6,0.5);
        \node[above] at (5.75,0.5) {\LARGE$s_1$};
        \draw[-,dotted] (4.5,-0.5) -- (4.5,1.0);
        \draw[-,dotted] (3,0) -- (3,0.5);
        \end{scope}
        \end{scope}
        
        \end{tikzpicture}
    \end{adjustbox}\\
    \vspace{3mm}
    
    \begin{adjustbox}{width=\textwidth}
        \begin{tikzpicture}
        \draw[|-|,dashed] (0,0.5) -- (0.5,0.5);
        \node[above] at (0.25,0.5) {\LARGE$s_2$};
        \draw[line width=4pt, red] (0.5,0) -- (1.5,0);
        \node[below] at (1,0) {$D_{1}$};

        \draw[-,dotted] (-0.0,-0.5) -- (-0.0,1.0);
        
        \draw[-,dotted] (0.5,0) -- (0.5,0.5);
        
        \draw[|-|,dashed] (1.5,0.5) -- (3.0,0.5);
        \node[above] at (2.25,0.5) {\LARGE$s_1$};
        \draw[line width=4pt, red] (3,0) -- (4,0);
        \node[below] at (3.5,0) {$D_{2}$};
        \draw[-,dotted] (1.5,0) -- (1.5,0.5);
        \draw[-,dotted] (3.0,0) -- (3.0,0.5);
        \draw[-,dotted] (4,0) -- (4,0.5);
        
        \begin{scope}[shift={(+4,0)}]
        \draw[|-|,dashed] (0.0,0.5) -- (0.5,0.5);
        \node[above] at (0.25,0.5) {\LARGE$s_2$};

        \draw[-,dotted] (-0.0,-0.5) -- (-0.0,1.0);

        \draw[line width=4pt, red] (0.5,0) -- (1.5,0);
        \node[below] at (1,0) {$D_{3}$};
        \draw[-,dotted] (0.5,0) -- (0.5,0.5);
        \node at (4.75,.25) {\dots};
        \end{scope}

        \begin{scope}[shift={(+7.5,0)}]
        \draw[line width=4pt, red] (-0.5,0) -- (0.5,0);
        \node[below] at (0.0,0) {$D_{4}$};

        \draw[|-|,dashed] (-2,0.5) -- (-0.5,0.5);
        \node[above] at (-1.25,0.5) {\LARGE$s_1$};

        \draw[-,dotted] (-0.5,0) -- (-0.5,0.5);
        \draw[-,dotted] (-2,0) -- (-2,0.5);
        
        \draw[-,dotted] (0.5,-0.5) -- (0.5,1.0);

        \draw[|-|,dashed] (2,0.5) -- (2.5,0.5);
        \node[above] at (2.25,0.5) {\LARGE$s_2$};
        \draw[line width=4pt, red] (2.5,0) -- (3.5,0);
        \node[below] at (3,0) {$D_{2m-1}$};
        \draw[-,dotted] (2,0) -- (2,0.5);
        \draw[-,dotted] (3.5,0) -- (3.5,0.5);
        \draw[-,dotted] (2.5,0) -- (2.5,0.5);
\begin{scope}[shift={(+2.5,0)}]
        \draw[|-|,dashed] (1,0.5) -- (2.5,0.5);
        \node[above] at (1.75,0.5) {\LARGE$s_1$};
        \draw[line width=4pt, red] (2.5,0) -- (3.5,0);
        \node[below] at (3,0) {$D_{2m}$};

        \draw[-,dotted] (2.5,0) -- (2.5,0.5);
        \draw[-,dotted] (3.5,-0.5) -- (3.5,1.0);
        \draw[-,dotted] (-0.5,-0.5) -- (-0.5,1.0);

        \end{scope}
        \end{scope}
        
\begin{scope}[shift={(-4,0)}]

        \draw[|-|,dashed] (0,0.5) -- (1.5,0.5);
        \node[above] at (0.75,0.5) {\LARGE$s_1$};
        \draw[line width=4pt, blue] (1.5,0) -- (2.5,0);
        \node[below] at (2,0) {$D_{-1}$};
        \draw[-,dotted] (1.5,0) -- (1.5, 0.5);
        
        \draw[|-|,dashed] (2.5,0.5) -- (3,0.5);
        \node[above] at (2.75,0.5) {\LARGE$s_2$};
        \draw[line width=4pt] (3,0) -- (4,0);
        \node[below,fill=white] at (3.5,-0.1) {$D_{0}$};
        \draw[-,dotted] (1.5,0) -- (1.5,0.5);
        \draw[-,dotted] (3,0) -- (3,0.5);
        \draw[-,dotted] (4,0) -- (4,0.5);
        \draw[-,dotted] (2.5,0) -- (2.5,0.5);

        \draw[-,dotted] (3,-0.5) -- (3,1.0);
        \draw[-,dotted] (-1,-0.5) -- (-1,1.0);
        \end{scope}
        
        \begin{scope}[shift={(-8.5,0)}]
        \node at (-1.25, 0.25) {\dots};
        \draw[line width=4pt, blue] (2.0,0) -- (3.0,0);
        \node[below] at (2.5,0) {$D_{-3}$};

        \draw[line width=4pt, blue] (-0.5,0) -- (0.5,0);
        \node[below] at (0,0) {$D_{-4}$};
        \draw[-,dotted] (-0.5,-0.5) -- (-0.5,1.0);
        \draw[-,dotted] (0.5,0.0) -- (0.5,0.5);
        \draw[-,dotted] (2,0.0) -- (2,0.5);

        \draw[|-|,dashed] (0.5,0.5) -- (2,0.5);
        \node[above] at (1.25,0.5) {\LARGE$s_1$};

        \draw[|-|,dashed] (3.0,0.5) -- (3.5,0.5);
        \node[above] at (3.25,0.5) {\LARGE$s_2$};
        \draw[line width=4pt, blue] (3.5,0) -- (4.5,0);
        \node[below] at (4,0) {$D_{-2}$};
        \draw[-,dotted] (3.0,0) -- (3.0,0.5);
        \draw[-,dotted] (3.5,0) -- (3.5,0.5);
        \draw[-,dotted] (4.5,0) -- (4.5,0.5);
        \end{scope}

        \begin{scope}[shift={(-15.5,0)}]
        \draw[line width=4pt, blue] (1,0) -- (2,0);
        \node[below] at (1.5,0) {$D_{-2m}$};

        \node[below] at (-1.5,0.5) {\huge Arrangement $B$: };
        \draw[|-|,dashed] (2,0.5) -- (3.5,0.5);
        \node[above] at (2.75,0.5) {\LARGE $s_1$};
        \draw[line width=4pt, blue] (3.5,0) -- (4.5,0);
        \node[below] at (4,0) {$D_{-2m+1}$};
        \draw[-,dotted] (2,0) -- (2,0.5);
        \draw[-,dotted] (4.5,0) -- (4.5,0.5);
        \draw[-,dotted] (3.5,0) -- (3.5,0.5);
\begin{scope}[shift={(+2,0)}]
        \draw[|-|,dashed] (2.5,0.5) -- (3,0.5);
        \node[above] at (2.75,0.5) {\LARGE $s_2$};
        \draw[-,dotted] (3,-0.5) -- (3,1.0);
        \draw[-,dotted] (-1,-0.5) -- (-1,1.0);

        \draw[-,dotted] (3,0) -- (3,0.5);
        \end{scope}
        \end{scope}
        
        \end{tikzpicture}
    \end{adjustbox}
    \caption{Finite resonator chain, consisting of $N_{tot} = 4m + 1$ resonators in total, all having length $\ell = 1$. Both resonator arrangements are mirror symmetric around $D_0$. Arrangement $A$ has spacings satisfying $s_1 < s_2$, whereas Arrangement $B$ has spacings satisfying $s_1 > s_2$.}
    \label{fig: geometrical defect}
\end{figure}
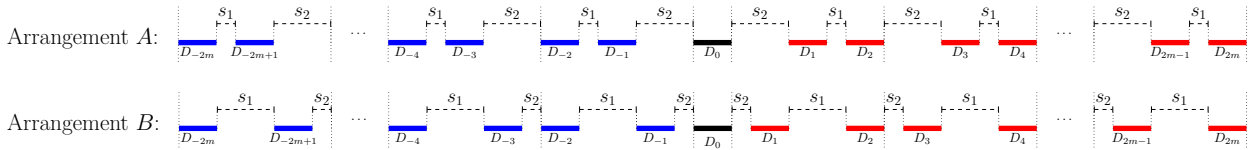

It was shown in~\cite{ammari2024exponentiallypreprint} that the capacitance matrix for a dimer chain, as illustrated in Figure~\ref{fig: geometrical defect}, is a mirror-symmetric tridiagonal $2$-Toeplitz matrix with perturbed diagonal corners, explicitly given by

\setcounter{MaxMatrixCols}{20}
\begin{equation}
\label{eq: strucutre capacitance matrix}
    C = \scalebox{0.7}{$\begin{pNiceMatrix}
    \Block[draw,fill=blue!30,rounded-corners]{6-6}{} \tilde{\alpha} & \beta_{1} &&&&&&&&&\\
\beta_{1} & \alpha & \beta_{2}  &&&&&&&&\\
& \beta_{2} & \alpha & \beta_{1}  &&&&&&&\\
       && \sddots     & \sddots     & \sddots     &&&&&&\\
       &&& \beta_{2} & \alpha & \beta_{1}  &&&&&\\
       &&&& \beta_{1} & \alpha & \beta_{2}  &&&&\\
       &&&&& \beta_{2} & \eta & \beta_{2}  &&&\\
       &&&&&& \beta_{2} & \alpha \Block[draw,fill=red!30,rounded-corners]{6-6}{} & \beta_{1}  &&\\
       &&&&& && \beta_{1} & \alpha & \beta_{2}  &\\
       &&&&&&&& \sddots     & \sddots & \sddots     &\\
       &&&&&&&&& \beta_{1} & \alpha & \beta_{2}  \\
       &&&&&&&&& & \beta_{2} & \alpha & \beta_{1}  \\
       &&&&&&&&&&& \beta_{1} & \tilde{\alpha}
    \end{pNiceMatrix},$}
\end{equation}
where the entries are defined by the spacings of the resonators $s_1$ and $s_2$ as in Figure~\ref{fig: geometrical defect},
\begin{equation}
    \beta_1 = -s_1^{-1}, \beta_2 = -s_2^{-1}, \alpha = s_1^{-1} + s_2^{-1}, \eta = 2s_2^{-1}, \tilde{\alpha}  = s_1^{-1}.
\end{equation}
The wave speeds inside the resonators are defined in~\eqref{eq: wavedpeed definition} and encoded in the matrix $V$ as in~\eqref{eq: wave speeds}, and due to the local inversion symmetry it must hold $v_1 = v_2 \in \C$. It is straightforward to verify that the generalised capacitance matrix $\mathcal{C}$ defined in \eqref{eq: generalised capacitance} has the same structure as \eqref{def: twofold Toeplitz operator}, with the building block $\mathbf{T}_A$ whose symbol is given by
\begin{equation}
    f_A(z) = v_1^2\begin{pmatrix}
        \alpha & \beta_1 + \beta_2z\\
        \beta_1 + \beta_2/z & \alpha
    \end{pmatrix} \in \C^{2\times 2}.
\end{equation}
In addition, $f_A(z)$ is persymmetric and therefore, by Proposition \ref{prop: complex dimer topology}, the edge modes of $\mathbf{T}_A$ are topologically protected. By Theorem \ref{thm: spectral inclusion}, this protection extends to the interface modes.
The existence of an edge mode is characterised by Theorem \ref{thm: existence of edge modes}, and for dimers the condition is explicitly derived in Example \ref{Example: existence of edge modes}. The operator $\mathbf{T}_A$ possesses an edge mode if and only if $|\beta_1| < |\beta_2|$ or equivalently for spacings $s_1>s_2$, corresponding to Arrangement $B$ in Figure \ref{fig: geometrical defect}. The edge induced interface eigenvalue and mode are illustrated in Figure \ref{Fig: SSH complex wavespeed} (A) and (C).

The situation is different for a resonator chain corresponding to Arrangement~$A$ in Figure~\ref{fig: geometrical defect}. In this case, the building block $\mathbf{T}_A$ does not possess edge modes, see Example \ref{Example: existence of edge modes}, yet the composite structure still supports an interface mode arising from a matched eigenvalue at the interface, as predicted by Theorem~\ref{thm: spectral inclusion}. The key distinction from Arrangement~$B$ is that this interface mode lacks topological protection, since it originates from an eigenvalue matched interface rather than from a topologically protected edge mode.

\begin{figure}[htb]
    \centering
    \subfloat[][Computations performed for $s_1 = 2, s_2 = 1$.]%
    {\includegraphics[width=0.45\linewidth]{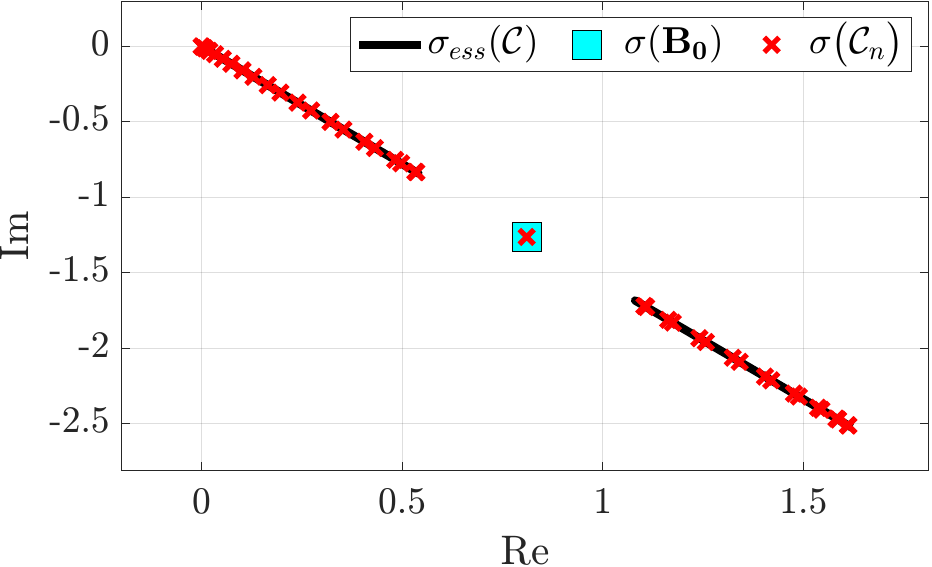}}\quad
    \subfloat[][Computations performed for $s_1 = 1, s_2 = 2$.]%
    {\includegraphics[width=0.45\linewidth]{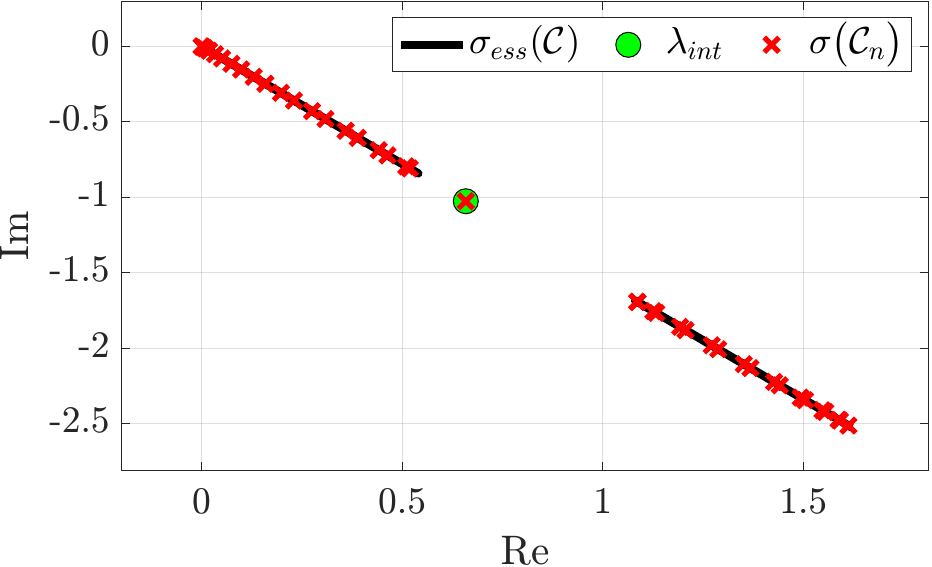}}\\
    \vspace{3mm}
    \subfloat[][Interface Dipole mode associated to (A).]%
    {\includegraphics[width=0.45\linewidth]{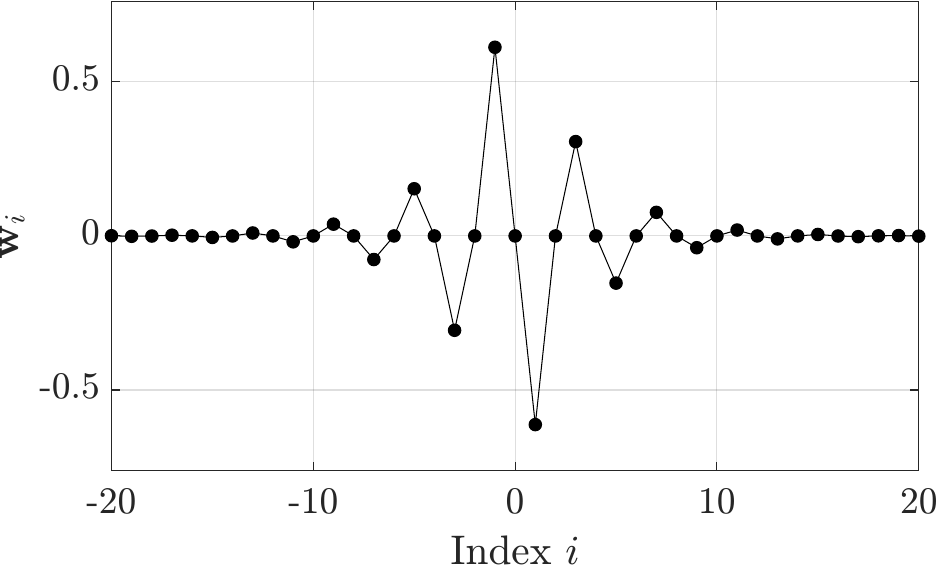}}\quad
    \subfloat[][Interface Monopole mode associated to (B).]%
    {\includegraphics[width=0.45\linewidth]{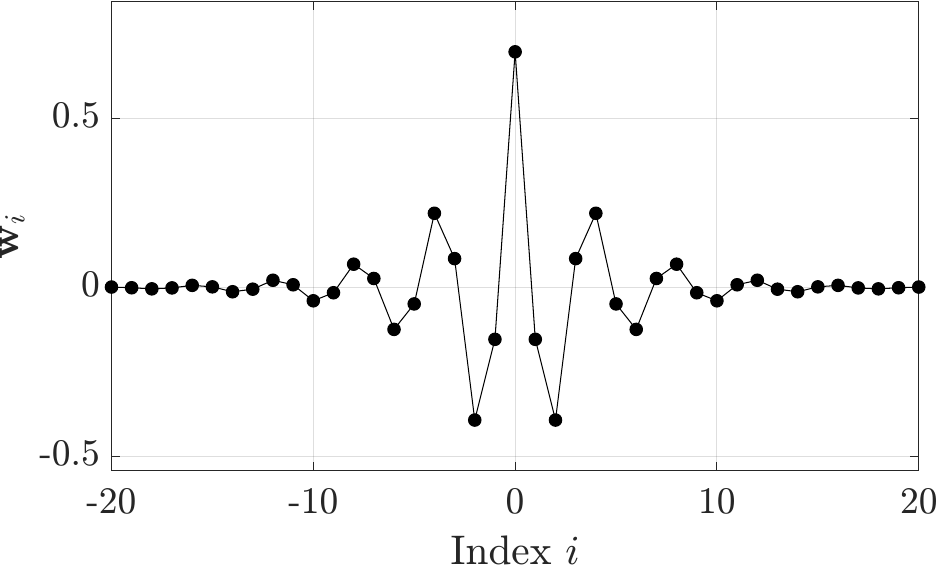}}
    \caption{ In Figure (A) we observe an edge induced interface mode. This is due to the fact that for $s_1 > s_2$ the one-sided Toeplitz matrix supports an edge mode, which by Theorem \ref{thm: spectral inclusion} carries over to the interface structure. Figure (B) the matched eigenvalue interface as given by Theorem \ref{thm: refexion symmetric impedance matching}. In both simulations, we considered a complex wave speed $v = e^{-\i}\approx0.54- 0.85\i$ in a chain comprising $N = 41$ resonators.}
    \label{Fig: SSH complex wavespeed}
\end{figure}

To illustrate this distinction between Arrangement $A$ and $B$ illustrated in Figure \ref{fig: geometrical defect}, we consider a compact perturbation consisting of changing only the two resonator spacings adjacent to the interface to $s_{\mathrm{int}}$, that is, the spacings separating $D_{-1}$ from $D_{0}$ and $D_{0}$ from $D_{1}$.
The interface mode in Figure \ref{Fig: SSH complex wavespeed} (B) is given by a matched interface eigenvalue given by roots to \eqref{eq: impedance analogue},
\begin{equation}
    F(\lambda) = \frac{2q^2z_1(\lambda)\mathbf{v}_1}{b_1\mathbf{v}_k} - \lambda + \eta = \frac{2}{s_{int}}\left(-\frac{s_{int}z_1(\lambda)\mathbf{v}_k}{s_2\mathbf{v}_1} + 1\right) - \lambda.
\end{equation}
Since $\lambda$ lies within the bandgap, $\mathrm{Im}(F(\lambda)) = 0$ and $\mathrm{Re}(F(\lambda))$ is a bounded function on a compact domain, one can always find a value of $s_{\mathrm{int}}$ for which $F(\lambda)$ has no roots. The interface mode can therefore be eliminated without closing the spectral gap, confirming that it is not topologically protected. We illustrate this behaviour numerically in Figure \ref{Fig: Not topological} (B).

\begin{figure}[htb]
    \centering
    \subfloat[][Computations performed for $s_1 = 2, s_2 = 1$.]%
    {\includegraphics[width=0.45\linewidth]{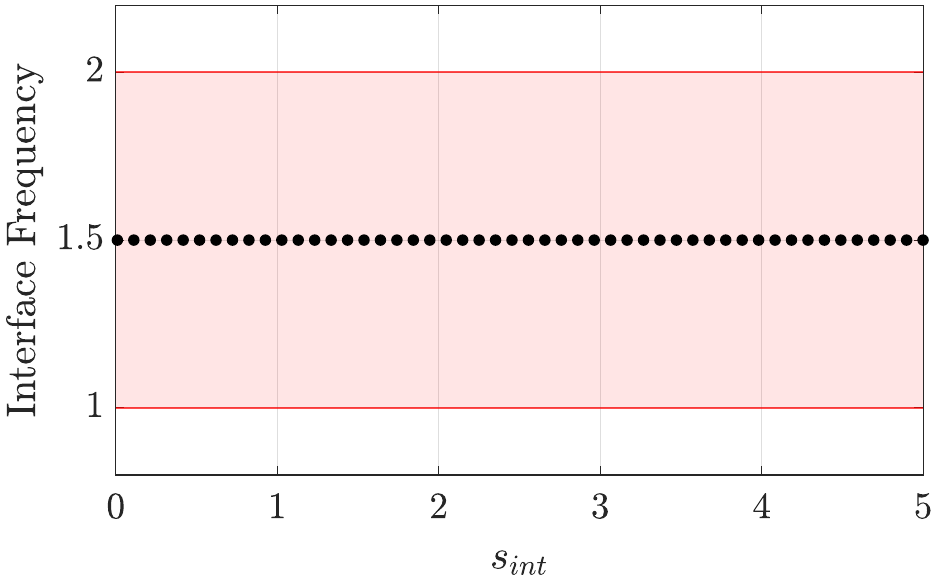}}\quad
    \subfloat[][Computations performed for $s_1 = 1, s_2 = 2$.]%
    {\includegraphics[width=0.45\linewidth]{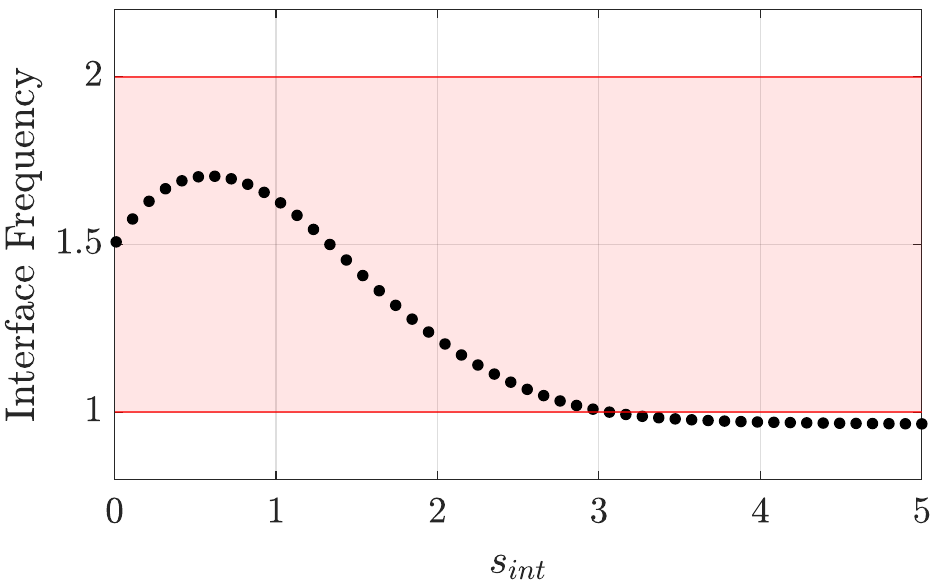}}
    \caption{Figure~(A) illustrates the topological protection of the interface mode, which is unaffected under a compact perturbation. In Figure~(B), the interface mode is not topologically protected, and a small perturbation pushes it into the essential spectrum. The bandgap is shaded in red.}
    \label{Fig: Not topological}
\end{figure}

\begin{figure}[htb]
    \centering
    \subfloat[][Computations performed for $s_1 = 2, s_2 = 1$.]%
    {\includegraphics[width=0.45\linewidth]{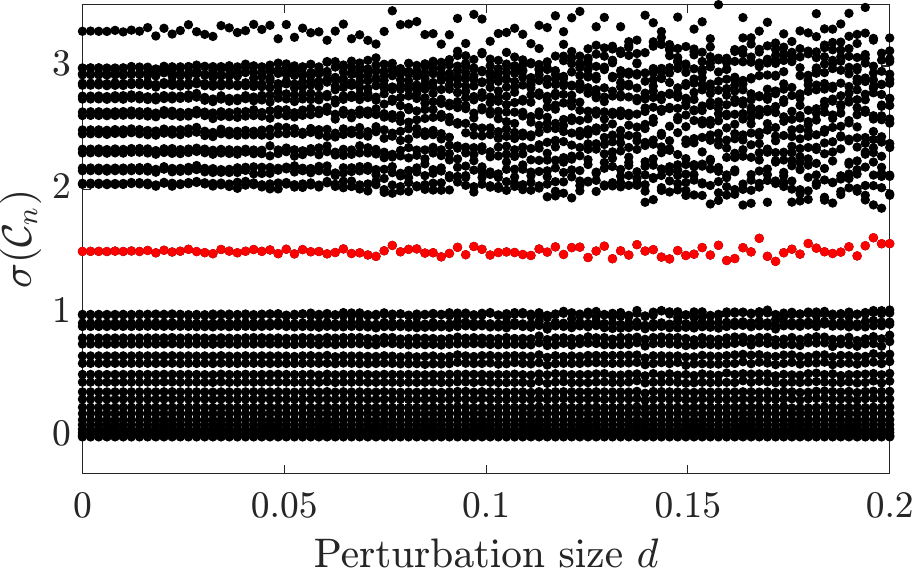}}\quad
    \subfloat[][Computations performed for $s_1 = 1, s_2 = 2$.]%
    {\includegraphics[width=0.45\linewidth]{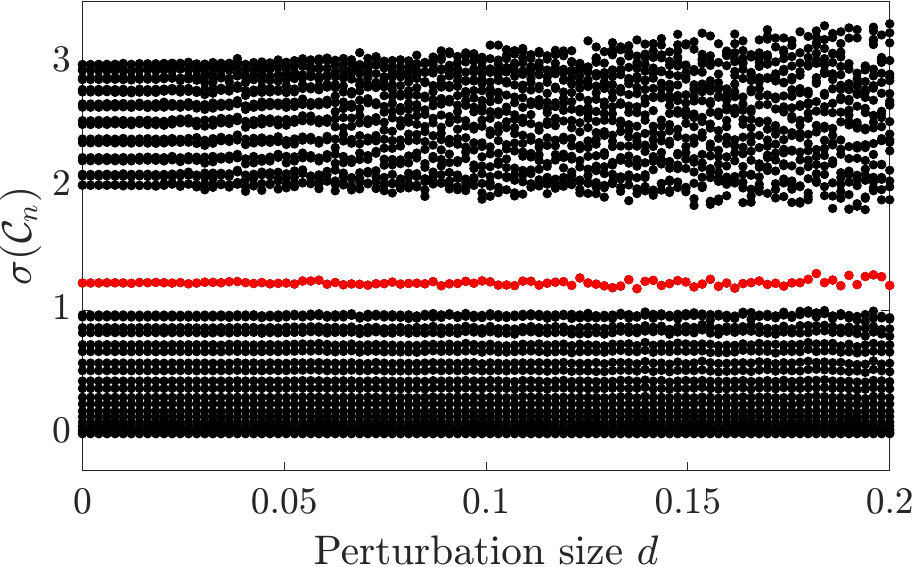}}
    \caption{Both structures exhibit the same stability with resect to perturbations in the resonator spacings, which is due to Weyl's Theorem of spectral stability \cite[Proposition 9]{ammari2024exponentiallypreprint}. The spectrum is shown for each perturbation size.}
    \label{Fig: Stability}
\end{figure}

Comparing with the discrete lattice model studied above reveals some crucial differences. It has been consistently shown in~\cite{10.1098/rspa.2022.0675} that impedance-matched interfaces, which are topologically robust in the continuum setting~\cite{10.1098/rspa.2023.0533}, lose this robustness in the discrete case. On the other hand, edge modes in semi-infinite structures are topologically protected in the discrete setting under mild symmetry assumptions (see Theorem~\ref{thm: inversion symmetry topological}). In contrast, one-sided semi-infinite mirror-symmetric photonic crystals are known not to possess edge 
modes~\cite{10.1098/rspa.2023.0533}.

\section{Connection between Continuum and Discrete models}\label{sec: Continuum Limit}
Having outlined the key differences between the discrete and continuum settings, we now relate the two by considering the continuum limit of a finite difference discretisation. We will study the same differential operator as in Section \ref{Sec: resonator chains} that is \eqref{eq: differential operator}. The discretised operator is denoted by $\mathbf{L}_{\Delta x}$, for a uniform step size $\Delta x = 1/k$ within the unit cell, it holds that
\begin{equation}
    -\frac{1}{\mu_0}\frac{\partial}{\partial x} \left( \frac{1}{\varepsilon(x)} \frac{\partial u}{\partial x}\right) \approx \mathbf{L}_{\Delta x} \mathbf{u}. 
\end{equation}
The permittivity $\varepsilon \colon \R \to \R$  is assumed in this section to be real-valued and piecewise constant in the resonator domain $D := \bigcup_{i=1}^N D_i$ as in \eqref{eq: piecewise permittivity}. 
The finite difference matrix $\mathbf{L}_{\Delta x} \in \R^{k\times k}$
is extended in space through the Floquet-Bloch quasiperiodic condition which introduces a twist in the antidiagonal corner entries, resulting in a symbol $f(z) \in \mathbb{C}^{k \times k}$ of the form \eqref{def: symbol of operator}.
In this setting, the leading principal submatrix $\mathbf{B}_0$ of $f(z)$ as  defined in \eqref{def: B0} can be interpreted as a discretisation of the unit cell, with one grid point of size $\Delta x$ omitted. Therefore, in the continuum limit $\Delta x \to 0$,
\begin{equation}
    \sigma(\mathbf{B}_0) \to \sigma\big(f(\mathbb{T})\big),
\end{equation}
in the Hausdorff sense, indicating that edge modes vanish in the continuum limit. This behaviour is rigorously established in \cite[Theorem 4.4]{10.1098/rspa.2023.0533} and numerically observed for relatively small $k$ in Figure \ref{Fig: Continuum Limit 1}: for any small $\Delta x$, the discrete problem admits edge states while the continuous problem with Dirichlet boundary conditions lack edge states.  

Furthermore, as we have seen in Figure \ref{Fig: Not topological}, the frequency-matched interface modes are not topologically protected in the discrete model. This is due to the fact that the roots of $F(\omega^2)$ defined in \eqref{eq: impedance analogue} may disappear under finite perturbations of the interface coupling parameters $\eta$ and $q$.
However, in the continuum limit $\Delta x \to 0$, the function $F(\omega^2)$ diverges as $\omega^2$ approaches the essential spectrum. This is in agreement with the surface-impedance function studied in \cite{10.1098/rspa.2023.0533}. Consequently, finite perturbations in $\eta$ and $q$ do not affect the existence of interface states, as a root to $F(\omega^2)$ is guaranteed by the intermediate value theorem. Hence, frequency-matched interface states become topologically protected in the continuum limit. This phenomenon is illustrated numerically in Figure \ref{Fig: Continuum Limit 1} and has been systematically studied in \cite{10.1098/rspa.2023.0533, 10.1098/rspa.2022.0675, alexopoulos2024topologicalinterfacemodessystems}.
\begin{figure}[htb]
    \centering
    \subfloat[][Bandgap between the first two spectral bands.]%
    {\vspace{0.7mm}\includegraphics[width=0.45\linewidth]{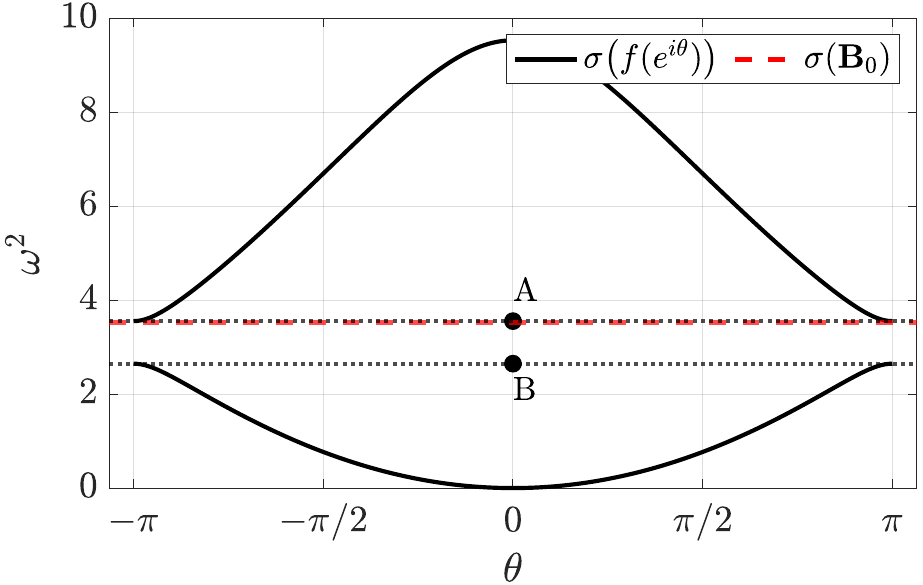}}\quad
    \subfloat[][Frequency function $F(\omega^2)$ over the first bandgap.]%
    {\includegraphics[width=0.48\linewidth]{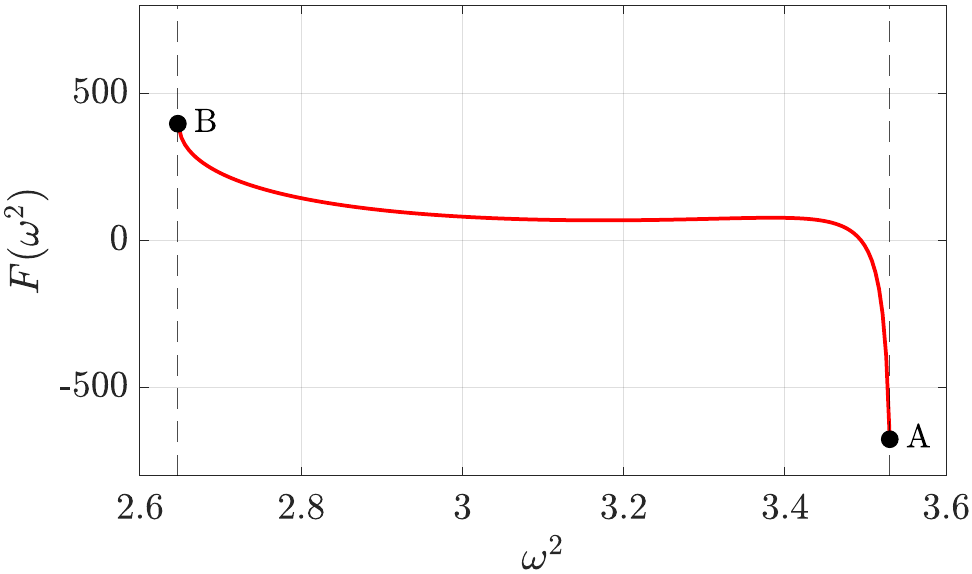}}
    \caption{Symbol function for the finite difference method $f(z)$. Computation performed for a dimer chain and using $k = 20$ discretisation points. The edge frequencies $\sigma(\mathbf{B}_0)$, asymptotically vanish into the essential spectrum as $k\to \infty$, indicating that there are no Dirichlet edge modes in the continuum \cite[Theorem 4.4.]{10.1098/rspa.2023.0533}. Contrary to the lattice model depicted in Figure \ref{Fig: Impedance matched}, in the continuum limit the branches of $F(\omega^2)$ diverge to $\pm \infty$ at $A$ and $B$ respectively. The intermediate value theorem therefore guarantees the existence of a topologically protected interface at $F(\omega^2) = 0$ in the continuum limit as described in \cite{10.1098/rspa.2023.0533}.}
    \label{Fig: Continuum Limit 1}
\end{figure}

\section{Aperiodic tight binding interface Hamiltonian}\label{Sec: Aperiodic Dimers}
Coburn's lemma reveals that the zero-energy interface mode supported by a $2k$-periodic SSH structure is immune to strong disorder.  Unlike Weyl's theorem on spectral stability, which only guarantees eigenvalue shifts bounded by the operator norm of the perturbation (see Figure \ref{Fig: Stability}), we show that the interface eigenvalue remains exactly unchanged under disorder and arbitrary perturbations. We further investigate the asymptotic localisation properties of this interface mode. Heuristically, the localisation strength is expected not to change, since the interface eigenvalue does not change relatively in the bandgap \cite{debruijn2024complexbandstructuresubwavelength, debruijn2025complexbrillouinzonelocalised}. However, the perturbation considered here is global rather than compact and, therefore, it is not immediately clear how the bandgap itself is affected by the disorder.
Concretely, we consider a structure in which the hopping amplitudes are perturbed, modelling imperfections such as manufacturing tolerances or material defects that arise in practical physical systems. The resulting aperiodic coupling between the resonators is illustrated in Figure~\ref{fig: Mirrored SSH Chain}.

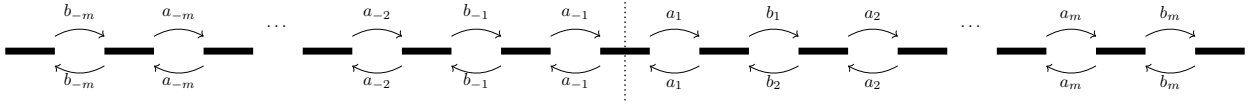
\begin{figure}[h!]
    \centering
    \begin{adjustbox}{width=\textwidth}
        \begin{tikzpicture}
        \draw[-,thick,dotted] (-.5,-1) -- (-.5,1);
        \draw[->] (2,0.3) to[bend left] (3,0.3);
        \draw[<-] (2,-0.3) to[bend right] (3,-0.3);
        \node[above] at (2.5,0.5)  {$b_1$};
        \node[above] at (2.5,-0.9) {$b_2$};
        
        \draw[line width=4pt]  (1,0) -- (2,0);

        \draw[->] (0,0.3)  to[bend left]  (1,0.3);
        \draw[<-] (0,-0.3) to[bend right] (1,-0.3);
        \node[above] at (0.5,0.5)  {$a_1$};
        \node[above] at (0.5,-0.9) {$a_1$};

        \draw[line width=4pt]  (3,0) -- (4,0);

        \draw[->] (4,0.3) to[bend left] (5,0.3);
        \draw[<-] (4,-0.3) to[bend right] (5,-0.3);
        \node[above] at (4.5,0.5) {$a_2$};
        \node[above] at (4.5,-0.9) {$a_2$};

        \begin{scope}[shift={(+4,0)}]
        
        \draw[line width=4pt] (1,0) -- (2,0);

        \node at (2.5,.5) {\dots};
        \end{scope}

        \begin{scope}[shift={(+6,0)}]
        \draw[line width=4pt]  (1,0) -- (2,0);
        
        \draw[->] (2,0.3) to[bend left] (3,0.3);
        \draw[<-] (2,-0.3) to[bend right] (3,-0.3);
        \node[above] at (2.5,0.5) {$a_m$};
        \node[above] at (2.5,-0.9) {$a_m$};
        
        \draw[line width=4pt]  (3,0) -- (4,0);

\begin{scope}[shift={(+2,0)}]
        \draw[->] (2,0.3) to[bend left] (3,0.3);
        \draw[<-] (2,-0.3) to[bend right] (3,-0.3);
        \node[above] at (2.5,0.5)  {$b_m$};
        \node[above] at (2.5,-0.9) {$b_m$};
        
        \draw[line width=4pt]  (3,0) -- (4,0);
  
        \end{scope}
        \end{scope}
        
\begin{scope}[shift={(-4,0)}]

        \draw[->] (2,0.3) to[bend left] (3,0.3);
        \draw[<-] (2,-0.3) to[bend right] (3,-0.3);
        \node[above] at (2.5,0.5) {$a_{-1}$};
        \node[above] at (2.5,-0.9) {$a_{-1}$};

        \draw[->] (0,0.3) to[bend left] (1,0.3);
        \draw[<-] (0,-0.3) to[bend right] (1,-0.3);
        \node[above] at (0.5,0.5) {$b_{-1}$};
        \node[above] at (0.5,-0.9) {$b_{-1}$};
        
        \draw[line width=4pt]  (1,0) -- (2,0);

        \draw[line width=4pt]  (3,0) -- (4,0);

        \end{scope}
        
        \begin{scope}[shift={(-8,0)}]
        \node at (.5,.5) {\dots};
        \draw[line width=4pt]  (1,0) -- (2,0);
        
        \draw[->] (2,0.3) to[bend left] (3,0.3);
        \draw[<-] (2,-0.3) to[bend right] (3,-0.3);
        \node[above] at (2.5,0.5) {$a_{-2}$};
        \node[above] at (2.5,-0.9) {$a_{-2}$};
        
        \draw[line width=4pt]  (3,0) -- (4,0);
        
        \end{scope}

        \begin{scope}[shift={(-14,0)}]
        \draw[line width=4pt]  (1,0) -- (2,0);

        \draw[->] (2,0.3) to[bend left] (3,0.3);
        \draw[<-] (2,-0.3) to[bend right] (3,-0.3);
        \node[above] at (2.5,0.5) {$b_{-m}$};
        \node[above] at (2.5,-0.9) {$b_{-m}$};
        
        \draw[line width=4pt]  (3,0) -- (4,0);

\begin{scope}[shift={(+2,0)}]
        \draw[->] (2,0.3) to[bend left] (3,0.3);
        \draw[<-] (2,-0.3) to[bend right] (3,-0.3);
        \node[above] at (2.5,0.5) {$a_{-m}$};
        \node[above] at (2.5,-0.9) {$a_{-m}$};
        
        \draw[line width=4pt]  (3,0) -- (4,0);

        \end{scope}
        \end{scope}
        
        \end{tikzpicture}
    \end{adjustbox}
    \caption{SSH chain with an interface in the middle and aperiodic coupling between the resonators.}
    \label{fig: Mirrored SSH Chain}
\end{figure}

The tight binding Hamiltonian corresponding to the structure is tridiagonal, symmetric with aperiodic off-diagonal coefficients. The symbol function, which varies between unit cells and is defined for the $i$-th cell, is
\begin{equation}
    f_i(z) = \begin{pmatrix}
        0 & a_i + b_iz\\
        a_i + b_i/z & 0
    \end{pmatrix}.
\end{equation}
Although the essential spectrum~\eqref{eq: essential spectrum} associated to each symbol function $f_i$ may vary between unit cells, the leading sub-matrix $\mathbf{B}_0 = 0$ remains unchanged. For $\lambda \in \sigma(\mathbf{B}_0) = \{0\}$, Theorem \ref{cor: coburn eigenvector} implies that the normalised Bloch eigenvector for any symbol function $f_i(z_i) \in \mathbb{C}^{2 \times 2}$ is given by
\begin{equation}\label{eq: aperiodic bloch vector}
    \mathbf{v} = \begin{pmatrix}
        1 \\ 0
    \end{pmatrix}, \quad \forall\, i \in \mathbb{Z}.
\end{equation}
We now construct an eigenvector of the semi infinite Toeplitz operator generated by the symbol functions $f_i(z)$. The Bloch eigenvector associated with $\lambda\in\sigma(\mathbf{B}_0)=\{0\}$ given by \eqref{eq: aperiodic bloch vector} is aperiodically extended such that by Theorem \ref{cor: coburn eigenvector} the following relation is satisfied.

\begin{equation}\label{eq: aperiodic approximation}
\scalebox{0.68}{$
\begin{pNiceMatrix}[columns-width=1.5em, cell-space-limits=0.15em]
 
\Block[draw,rounded-corners]{2-2}{}
0 & a_{0} &
\Block[draw,rounded-corners]{2-2}{}
0 & 0 &  &  \\
a_{0} & 0 & b_{1} & 0 &  &  \\
\Block[draw,rounded-corners]{2-2}{}
0 & b_{0} &
\Block[draw,rounded-corners]{2-2}{}
0 & a_{1} &
\Block[draw,rounded-corners]{2-2}{}
0 & 0 \\
0 & 0 & a_{1} & 0 & b_{2} & 0 \\
  &  &
\Block[draw,rounded-corners]{2-2}{}
0 & b_{1} &
\Block[draw,rounded-corners]{2-2}{}
0 & a_{2} \\
  & & 0 & 0 & a_{2} & 0 & \smash{\ddots} \\
 & &   &   &       & {\ddots} & {\ddots}
\end{pNiceMatrix}$} \scalebox{0.7}{$ 
 \left(\begin{array}{c}
    \phantom{-}z_{0}\begin{pmatrix}
        1 \\ 0 \end{pmatrix} \\
    \phantom{-}z_{1}\begin{pmatrix}
        1 \\ 0 \end{pmatrix} \\
    \phantom{-}z_{2}\begin{pmatrix}
        1 \\ 0 \end{pmatrix} \\
        \vdots
\end{array}\right)$} = \scalebox{0.68}{$\begin{pNiceMatrix}[cell-space-limits=0.2em]
    0 \\
    a_{0} z_{0} + b_{1}z_{1} = 0\\
    0 \\
    a_{1} z_{1} + b_{2}z_{2} = 0\\
    0\\ 
    a_{2} z_{2} + b_{3}z_{3} = 0\\
    \vdots
\end{pNiceMatrix}$}.
\end{equation}
Consequently, it is straightforward to verify that the aperiodic extension coefficients $z_i$ can be computed recursively via
\begin{equation}
    z_i = -\frac{a_{i-1}z_{i-1}}{b_i}, \quad \forall i \in \N.
\end{equation}
By normalisation, we assume without loss of generality that $z_0 = 1$. Then it is not hard to see that
\begin{equation}\label{eq: random process}
    z_i= (-1)^i \prod_{k=0}^i \frac{a_k}{b_{k+1}}, \quad \text{ for } i > 0.
\end{equation}
In a similar fashion as in Theorem \ref{thm: refexion symmetric impedance matching} we will demonstrate that an interface mode associated to a structure such as in Figure \ref{fig: Mirrored SSH Chain} may be constructed using the eigenvectors of the doubly infinite Laurent operator given  by \eqref{eq: aperiodic approximation}.

\begin{theorem}\label{thm: disordered SSH interface}
  Let $\mathbf{u}$ and $\mathbf{s}$ be eigenvectors of the doubly infinite left and right structures, respectively. Then the composite material merged at an interface supports an interface eigenmode at the eigenvalue $\lambda = 0$, with corresponding eigenvector $\mathbf{w}$, which is given by,
    \begin{equation}
        \mathbf{w}_i = \begin{cases}
            a\vect{s}_{|i|},~&i < 0,\\
            0, ~&i = 0,\\
            \mathbf{u}_{i}, &i > 0.
        \end{cases}
    \end{equation}
\end{theorem}

\begin{proof}
    Suppose that $\mathbf{L}_A$ is a doubly infinite aperiodic dimer structure with zero on the diagonal such that $\mathbf{L}_A\mathbf{u} = 0$ which may be constructed as in \eqref{eq: aperiodic approximation} and let $\mathbf{A}$ be the truncated semi-infinite operator and $(\mathbf{u}_1, \mathbf{u}_2, \dots)^\top$ be the truncated doubly infinite vector $\mathbf{u}$. A similar construction may be carried out for $\mathbf{B}$, since $\mathbf{B}$ in general does not need to be a reflection of $\mathbf{A}$, the truncated eigenvector $(\mathbf{s}_1, \mathbf{s}_2, \dots)^\top$ also differs from $(\mathbf{u}_1, \mathbf{u}_2, \dots)^\top$. The matching condition at the interface may now be achieved as follows,
    \begin{equation}\label{eq: Dimer SSH matching}
         \scalebox{0.8}{$\left[
        \begin{pNiceMatrix}[columns-width=0em]
        \Block[draw,fill=blue!31,rounded-corners]{3-3}{} & & & &  \\
        & \mathbf{B} & & &  \\
        & &  & q &  \\
        & & q &  0 & p & & \\
        & & & p & \Block[draw,fill=red!31,rounded-corners]{3-3}{} & &  \\
        & & & & & \mathbf{A} &  \\
        & & & & & &  
        \end{pNiceMatrix} - \lambda \mathrm{Id} \right]$}\scalebox{0.7}{$
        \begin{pNiceMatrix}[columns-width=0em]
        a{\color{blue!80}\begin{pmatrix}
            \vdots \\
            \vect{s}_2 \\
            \vect{s}_1
        \end{pmatrix}}\\
        t\\
        {\color{red!80}\begin{pmatrix}
            \vect{u}_1 \\
            \vect{u}_2\\
            \vdots
        \end{pmatrix}}
        \end{pNiceMatrix}$} = \scalebox{0.7}{$\begin{pmatrix}
            \vdots \\
            0 \\
            qt - a\vect{s}_{-1} \\
            aq\vect{s}_1 +  q\vect{u}_1\\
            qt - \vect{u}_{-1} \\
            0 \\
            \vdots
        \end{pmatrix}$}
    \end{equation}
    The Bloch eigenvector in \eqref{eq: aperiodic bloch vector} leads to $\mathbf{u}_{-1} = z_{-1}\mathbf{v}_2 = 0$, and by an identical reasoning we also obtain $\mathbf{s}_{-1} = 0$. Consequently, it follows directly that $t = 0$. The middle equation in the residual \eqref{eq: Dimer SSH matching} is satisfied provided that
    \begin{equation}
        a = -\frac{q\vect{s}_{1}}{p\vect{u}_{1}},
    \end{equation}
    which completes the proof.
\end{proof}

We will now proceed to illustrate the asymptotic behaviour of the interface eigenmode. Note that in the fully  periodic case, that is, $a_i = a$ and $b_i =  b$, the aperiodic extension factors reduce to
\begin{equation}
    z_i = (-1)^i \prod_{k = 0}^i\frac{a_i}{b_{i+1}} =  (-1)^i \prod_{k = 0}^i\frac{a}{b} = \left(-\frac{a}{b}\right)^i
\end{equation}
Interface modes are characterised by being exponentially localised around the interface, requiring $|z_i| < 1$. It follows that any $\lambda \in \sigma(\mathbf{B}_0)$ supporting such a mode must lie within the asymptotic spectral gap.
Although the structure in~\eqref{eq: aperiodic approximation} is arranged in a $2 \times 2$-block form, the underlying Laurent operator need not be $2$-periodic, as the symbol functions $f_i$ are arbitrary.
Nevertheless, only even-periodic systems support exponentially localised interface modes.
Consider the diagonal matrix $\mathbf{D} = \operatorname{diag}(1, -1, 1, -1, \dots)$, which satisfies $\mathbf{D}^{-1} = \mathbf{D}$. The chiral symmetry
\begin{equation}
    \mathbf{D}\mathbf{A}\mathbf{D}^{-1} = \mathbf{D}\mathbf{A}\mathbf{D} = -\mathbf{A}
\end{equation}
immediately yields $\sigma(\mathbf{A}) = \sigma(-\mathbf{A})$, so the spectrum is symmetric about the origin. Since a $k$-Toeplitz operator has an essential spectrum consisting of $k$ bands, this symmetry forces the bands to appear in reflected pairs. When $k$ is odd, one band must remain unpaired and hence contain the origin.
Therefore, for odd valued $k$, $\lambda \in \sigma(\mathbf{B}_0)$ will be situated within the essential spectrum and the associated interface mode will not be exponentially localised.
Let us therefore assume for the remainder of the section that the structure is $2k$-periodic, such that
\begin{equation}\label{eq: 2k perturbed coefficients}
\scalebox{0.9}{$
    \begin{aligned}
        a_i &= a_1(1 + \varepsilon_{1_i}),\quad \varepsilon_{1_i} \sim \mathcal{U}(-d, d), \quad\dots, \quad  a_{i+k-2} &= a_k(1 + \varepsilon_{k_i}),\quad \varepsilon_{k_i} \sim \mathcal{U}(-d, d) \\
        b_i &= b_1(1 + \tau_{1_i}),\quad \tau_{1_i} \sim \mathcal{U}(-d, d),\quad \dots,\quad b_{i+k-2} &= b_k(1 + \tau_{k_i}),\quad \tau_{k_i} \sim \mathcal{U}(-d, d)\\
    \end{aligned}$}
\end{equation}
The uniform perturbations are chosen in such a way that they preserve the sign of the couplings. In Theorem \ref{thm: disordered SSH interface} we demonstrate that the interface eigenvalue $\lambda \in \sigma(\mathbf{B}_0) = \{0\}$ is robust to disorder, we will therefore venture onwards and characterise the eigenvector asymptotics in the presence of disorder.
We will now proceed to study the localisation strength of the interface eigenmode within the disordered system.
For a $2k$-periodic structure, we denote $\tilde{z}_i$ the multiplicative process over $k$ dimer cells so that it holds
\begin{equation}\label{eq: multiplicative k process}
    \tilde{z}_i = \prod_{j=0}^{ki}\frac{a_i}{b_{i+1}}.
\end{equation}

\begin{proposition}\label{prop: aperiodic decay rate 2k}
    Let $\tilde{z}_i$ be the $2k$-multiplicative process defined in \eqref{eq: multiplicative k process} with perturbed coupling strengths as in \eqref{eq: 2k perturbed coefficients}. Then the process $\tilde{z}_i$ almost surely decays exponentially, that is
    \begin{equation}
        \lim_{i\to\infty} \frac{1}{i}\ln(|\tilde{z}_i|) = \ln\left(\left|\frac{a_1\dots a_k}{b_1 \dots b_k}\right|\right) \quad\text{a.s}.
    \end{equation}
\end{proposition}

\begin{proof}
    From the definition of the multiplicative process $\tilde{z}_i$ in \eqref{eq: multiplicative k process} it follows that,
    \begin{equation}\label{eq: process asymp}
        \ln(|\tilde{z}_i|) = i\ln\left(\left|\frac{a_1\dots a_k}{b_1 \dots b_k}\right|\right) +\sum_{t=1}^k\left( \sum_{j = 0}^i\Big(\ln(1 + \varepsilon_{t_k}) - \ln(1+\tau_{t_k})\Big)\right).
    \end{equation}
    By the law of large numbers it holds that 
    \begin{equation}
        \frac{1}{i}\sum_{j = 0}^i\ln(1 + \varepsilon_{t_k}) \to \mathbb{E}\big[\ln(1+\varepsilon)\big]\quad \text{and}\quad
        \frac{1}{i}\sum_{j = 0}^i\ln(1 + \tau_{t_k}) \to \mathbb{E}\big[\ln(1+\tau)\big].
    \end{equation}
    Since the random variables $\varepsilon$ and $\tau$ follow the same distribution $\mathcal{U}(-d, d)$ their expectations agree. Therefore, as $i$ becomes large it holds that
    \begin{equation}
        \frac{1}{i}\ln(|\tilde{z}_i|) \to \ln\left(\left|\frac{a_1\dots a_k}{b_1 \dots b_k}\right|\right) \quad\text{almost surely,}
    \end{equation}
    which completes the proof.
\end{proof}

We will conclude this section by numerically illustrating our results. In Figure \ref{Fig: Disorder Spectrum} we illustrate the robustness of the interface eigenvalue to disorder, both in the case of real and complex symmetric coupling strengths. The robustness of the localisation properties is illustrated in Figure \ref{Fig: Disorder and decay rate}.

\begin{figure}[htb]
    \centering
    \subfloat[][Computation performed for $a = 1,~ b  = 2$ and $n = 99$.]%
    {\includegraphics[width=0.45\linewidth]{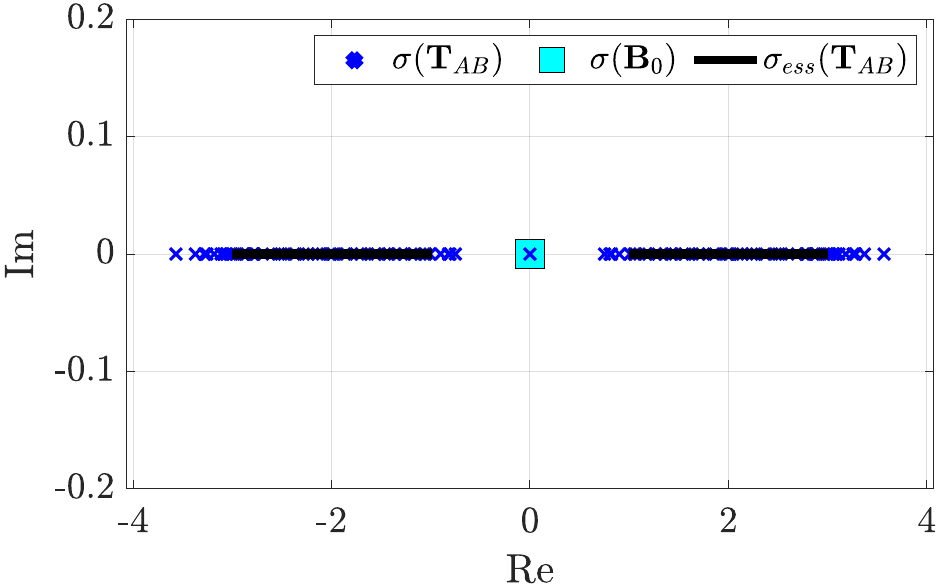}}\quad
    \subfloat[][Computation performed for $a = 1-0.3\i,~ b  = 2 + 1\i$ and $n = 99$.]%
    {\includegraphics[width=0.45\linewidth]{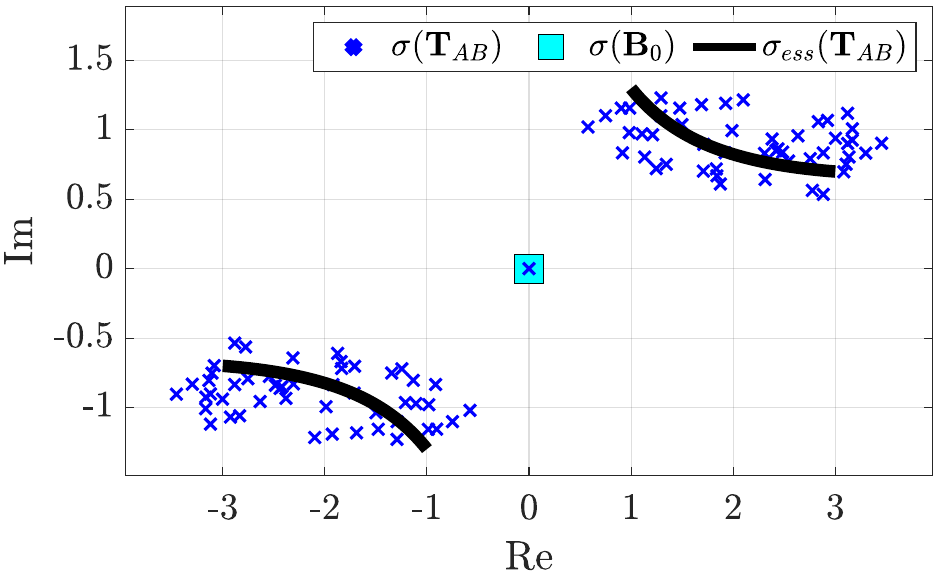}}
    \caption{As established in Theorem \ref{thm: disordered SSH interface}, interface eigenvalue is robust with respect to disorder and stays fixed exactly at $\lambda \in \sigma(\mathbf{B}_0)=\{0\}$. Computation performed for a disorder rate of $d = 0.4$.}
    \label{Fig: Disorder Spectrum}
\end{figure}

\begin{figure}[htb]
    \centering
    \subfloat[][Computation performed for $a = 1,~ b  = 2$ and $n = 39$.]%
    {\includegraphics[width=0.45\linewidth]{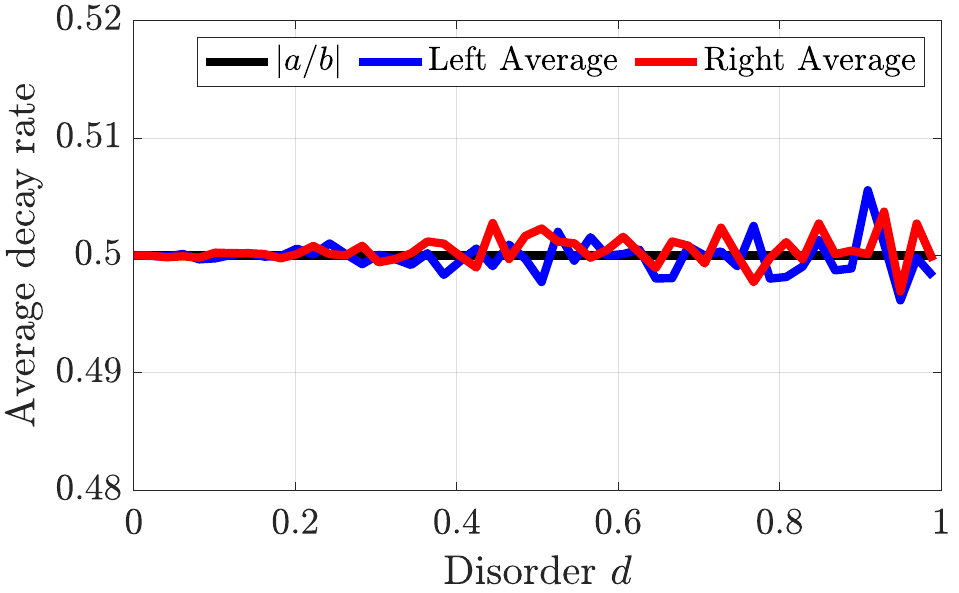}}\quad
    \subfloat[][Computation performed for $a = 1-0.3\i,~ b  = 2 + 1\i$ and $n = 39$.]%
    {\includegraphics[width=0.45\linewidth]{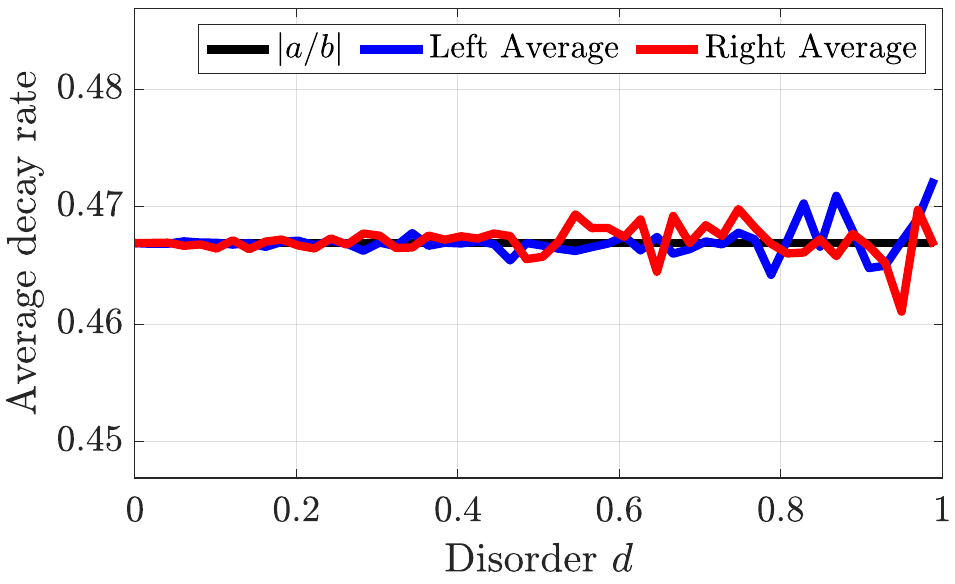}}
    \caption{As asserted by Proposition \ref{prop: aperiodic decay rate 2k}, the asymptotic decay rate is preserved under disorder. Average decay rate is computed over a total of $5000$ realisations.}
    \label{Fig: Disorder and decay rate}
\end{figure}

\begin{figure}[htb]
    \centering
    \subfloat[][Spectrum for a $4$-periodic coupling with a disorder of strength $d = 0.4$.]%
    {\includegraphics[width=0.45\linewidth]{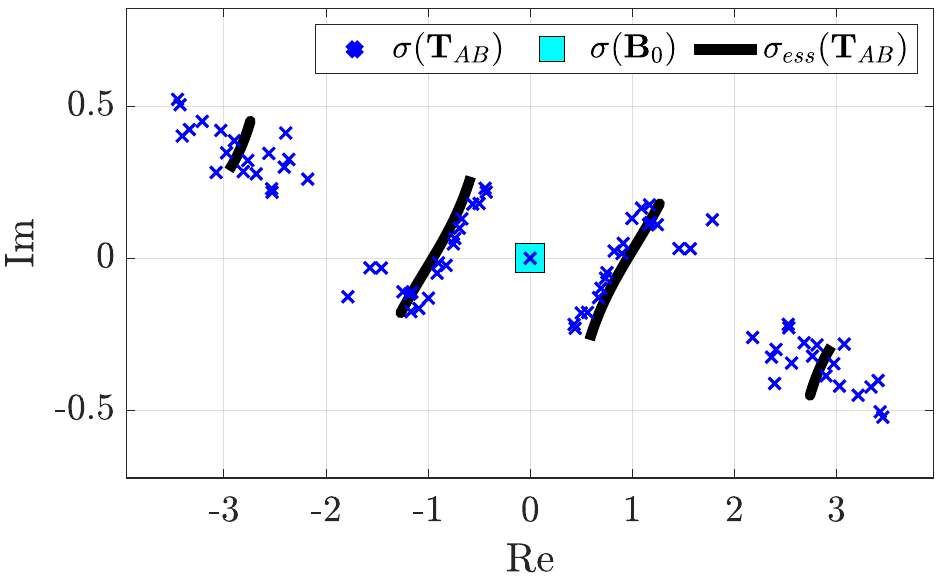}}\quad
    \subfloat[][Average of $1000$ realisations interface eigenmode localisation strengths at different disorder rates $d$.]%
    {\includegraphics[width=0.45\linewidth]{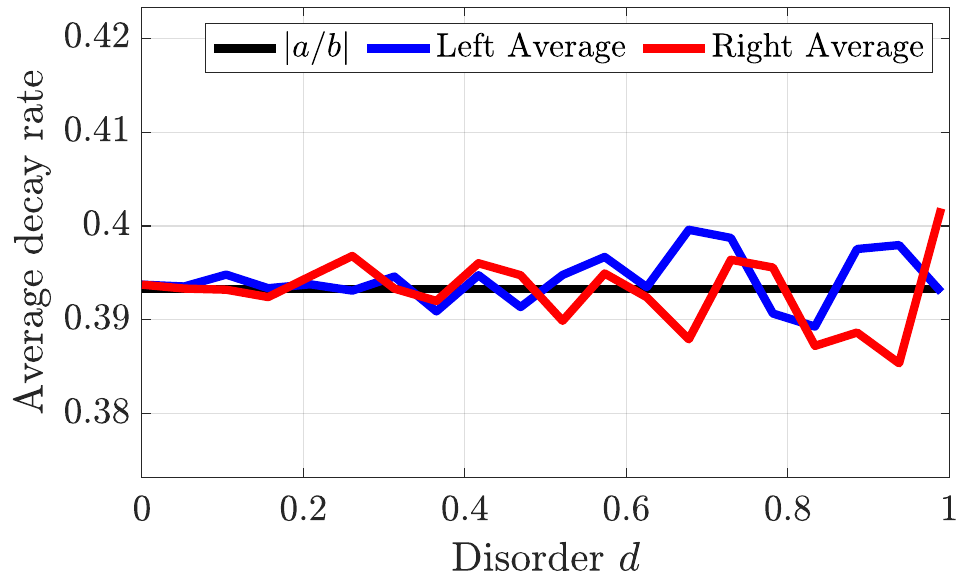}}
    \caption{Computation performed for  $a_1 = 2 + 0.4\i, b_1 = 0.5 + 0.3\i, a_3 = 1.3, b_2=1.8 - 0.2\i$ and $n = 79$.}
    \label{Fig: Disorder Spectrum and decay 4 Toeplitz}
\end{figure}

\section{Concluding Remarks}\label{Sec: discussion}
We have established a novel connection between Coburn’s Lemma for tridiagonal $k$-Toeplitz operators and the emergence of edge-induced interface modes. In this framework, the eigenvalues of the leading principal submatrix of the symbol function characterise the edge modes and can be computed explicitly with minimal effort. Since the result applies to general tridiagonal $k$-Toeplitz matrices, it may be interpreted as a form of bulk-boundary correspondence for multi-band lattice models.
We further identified conditions under which these edge modes are topologically protected, namely, when the symbol function exhibits local inversion symmetry. In addition, we presented a comprehensive analysis of tridiagonal interface operators, distinguishing between edge-induced modes and matched-eigenvalue interface modes.
These results were applied to damped subwavelength resonator chains and tight-binding Hamiltonians, demonstrating the existence of topologically protected interface modes. In particular, we establish their robustness under random perturbations and even under strong disorder.
Finally, the differences in interface states between continuum and lattice models were presented, and the transition in their behaviour from finite systems to the continuum limit was conceptually described.
As a future research direction, we plan to investigate the reality of the asymptotic open spectrum in the tridiagonal $k$-Toeplitz setting, in line with recent studies \cite{debruijn2026mathematicalfoundationgeneralisedbrillouin, giandinoto2024realityeigenvaluesbandedblock}.     

\section{Conflict of interest}
The authors have no competing interests to declare.

\section{Data availability} \label{Sec: Data availability}
The \texttt{Matlab} code for the numerical experiments developed in this work is openly available in the following repository:
\url{https://github.com/yannick2305/TopologicalInterfaces}.

\appendix
\section{Interface Toeplitz operator joined by coupling}\label{Sec: common coupling} Having discussed interface Toeplitz operators joined at a common interface, we now turn to the case where the building blocks $\mathbf{T}_A$ and $\mathbf{T}_B$ are coupled by a parameter $q$, as introduced in~\eqref{def: common coupling} below.
\begin{equation}\label{def: common coupling}
    \mathbf{T}_{AB} = \scalebox{0.8}{$\begin{pNiceMatrix}[columns-width=0em]
        \Block[draw,fill=blue!31,rounded-corners]{3-3}{} & & & &  \\
        & \mathbf{T}_B& & &  \\
        & &  & q &  \\
         & & q& \Block[draw,fill=red!31,rounded-corners]{3-3}{} & &  \\
         & & & & \mathbf{T}_A &  \\
         & & & & &  
        \end{pNiceMatrix}$}
\end{equation}
Matching the eigenvalues at the interface yields the following result.
\begin{theorem}\label{thm: common coupling}
    Suppose that $\mathbf{T}_{AB}$ is as in \eqref{def: common coupling}, let $f$ be the symbol function of $\mathbf{T}_A$ and let $\mathbf{v}$ be such that 
    \begin{equation}
        \Big[f\big(z(\lambda)\big)-\lambda \Id \Big]\mathbf{v} = 0,
    \end{equation}
    then $\lambda$ is an interface eigenvalue if
        \begin{equation}
        \frac{\mathbf{v}_2}{\mathbf{v}_1} = - \frac{aq + ({ \mathbf{T}_A})_{1,1} - \lambda}{({ \mathbf{T}_A})_{1,2}},\quad a = \pm 1.
    \end{equation}
    In this case, $\mathbf{T}_{AB}\vect{w} = \lambda \vect{w}$, where
    \begin{equation}
        \mathbf{w}_i = \begin{cases}
            a\left(\mathbf{R}\vect{u}\right)_{|i|}, & i \leq 0 \\
            \vect{u}_i, & i > 0
        \end{cases}
    \end{equation}
    where $\vect{u} = \big(z_1^0\vect{v}, z_1^1\vect{v}, z_1^2\vect{v}, \dots \big)$ is the quasiperiodically extended Bloch eigenvector.
    
\end{theorem}
\begin{proof}
    Let $\mathbf{u}$ be an eigenvector of the Laurent operator to the eigenvalue $\lambda$ such that $\mathbf{L}_A\mathbf{u} = \lambda \mathbf{u}$. Then the eigenvector is constructed as follows,
    \begin{equation}
        \mathbf{u} = \big(\dots, z_1^{-2}\vect{v}, z_1^{-1}\vect{v}, z_1^0\vect{v}, z_1^1\vect{v}, z_1^2\vect{v}, \dots \big)
    \end{equation}
    where $z_1$ is the root of the smallest magnitude of $\operatorname{det}\big(f(z)-\lambda\Id\big) = 0$ and $\big[f(z_1)-\lambda\Id\big]\vect{v} = 0$. The interface eigenvalue problem becomes
    \begin{equation}
        \scalebox{0.85}{$\left[\begin{pNiceMatrix}[columns-width=0em]
        \Block[draw,fill=blue!31,rounded-corners]{3-3}{} & & & &  \\
        & \mathbf{T}_B& & &  \\
        & &  & q &  \\
         & & q& \Block[draw,fill=red!31,rounded-corners]{3-3}{} & &  \\
         & & & & \mathbf{T}_A &  \\
         & & & & &  
        \end{pNiceMatrix} - \lambda \Id\right]$}
        \scalebox{0.75}{$\begin{pNiceMatrix}[columns-width=0em]
        a{\color{blue!80}\begin{pmatrix}
            \vdots \\
            \vect{u}_2 \\
            \vect{u}_1
        \end{pmatrix}}\\
       \phantom{a} {\color{red!80}\begin{pmatrix}
            \vect{u}_1 \\
            \vect{u}_2\\
            \vdots
        \end{pmatrix}}
        \end{pNiceMatrix}$} =\scalebox{0.75}{$\begin{pmatrix}
            \vdots \\
            0 \\
            ({ \mathbf{T}_B})_{1,2}
            aq\vect{u}_2 + \big(({\mathbf{T}_B})_{1,1} -\lambda\big)a\vect{u}_1 + q\vect{u}_1\\
            aq\mathbf{u}_1 + \big(({ \mathbf{T}_A})_{1,1}-\lambda\big)\vect{u}_{1} + ({ \mathbf{T}_A})_{1,2}\mathbf{u}_2 \\
            0 \\
            \vdots
        \end{pmatrix} $}.  
    \end{equation}
    Due to reflection symmetry, we have $({ \mathbf{T}_A})_{1,1} = ({\mathbf{T}_B})_{1,1}$ as well as $({\mathbf{T}_A})_{1,2} = ({ \mathbf{T}_B})_{1,2}$. The residual vanished precisely when
    \begin{equation}
        \begin{pmatrix}
            \big(({\scriptstyle \mathbf{T}_A})_{1,1} -\lambda\big)a + q & ({\scriptstyle \mathbf{T}_A})_{1,2}
            aq \\
            aq + \big(({\scriptstyle \mathbf{T}_A})_{1,1}-\lambda\big) & ({\scriptstyle \mathbf{T}_A})_{1,2}
        \end{pmatrix}\begin{pmatrix}
            \mathbf{u}_1\\
            \mathbf{u}_2
        \end{pmatrix} = \begin{pmatrix}
            0 \\ 0
        \end{pmatrix}
    \end{equation}
    which is satisfied for $\alpha = \pm 1$, independently of the other entries. Moreover, it must hold that
    \begin{equation}\label{eq: matching condition via coupling}
        \frac{\mathbf{u}_2}{\mathbf{u}_1} = - \frac{aq + ({ \mathbf{T}_A})_{1,1} - \lambda}{({ \mathbf{T}_A})_{1,2}},\quad a = \pm 1.
    \end{equation}
    Since $\mathbf{u_1}$ and $\mathbf{u_2}$ are the entries in the first unit cell, they agree with the first two entries of the Bloch vector $\mathbf{v}$, which completes the proof of the result.
\end{proof}

We would like to emphasise that the matching condition in \eqref{eq: matching condition via coupling} depends on the parameter $q$. Consequently, a compact perturbation of the operator, which would change the couplings $q$ could potentially cause the interface mode to vanish. We will proceed to numerically illustrate the results of the interface modes in the case of a dimer system in Figure \ref{Fig: Interface Twofold Dimer Common coupling}.

\begin{figure}[htb]
    \centering
    \subfloat[][Computation performed for $a_1 = 1.2, a_2 = 1, b_1 =c_1 =1.8-0.8\i$ and $b_2 = c_2 = 3.2-1\i$.]%
    {\includegraphics[width=0.465\linewidth]{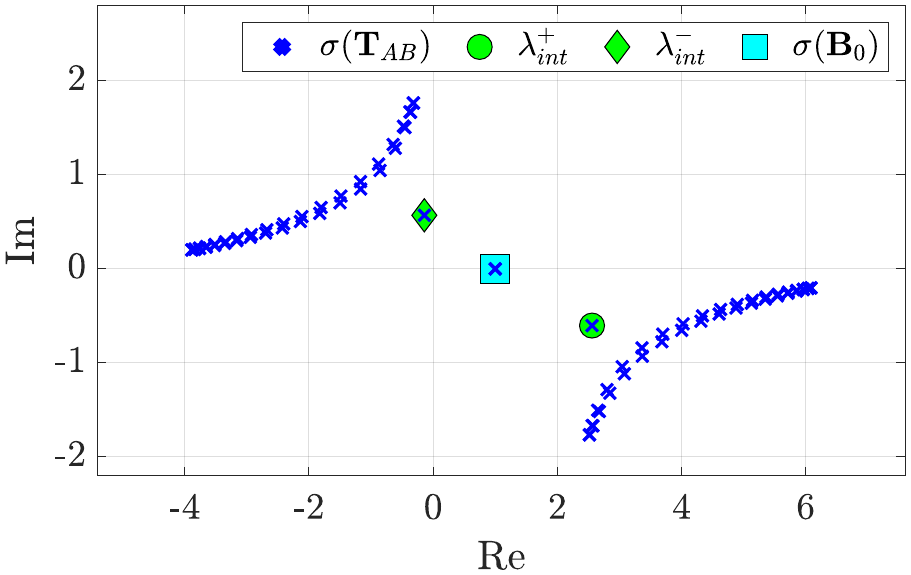}}\quad
    \subfloat[][Computation performed for $a_1 = 1.2, a_2 = 1, b_1 =c_1 =1.8-0.8\i$ and $b_2 = c_2 = 0.5+\i$.]%
    {\includegraphics[width=0.45\linewidth]{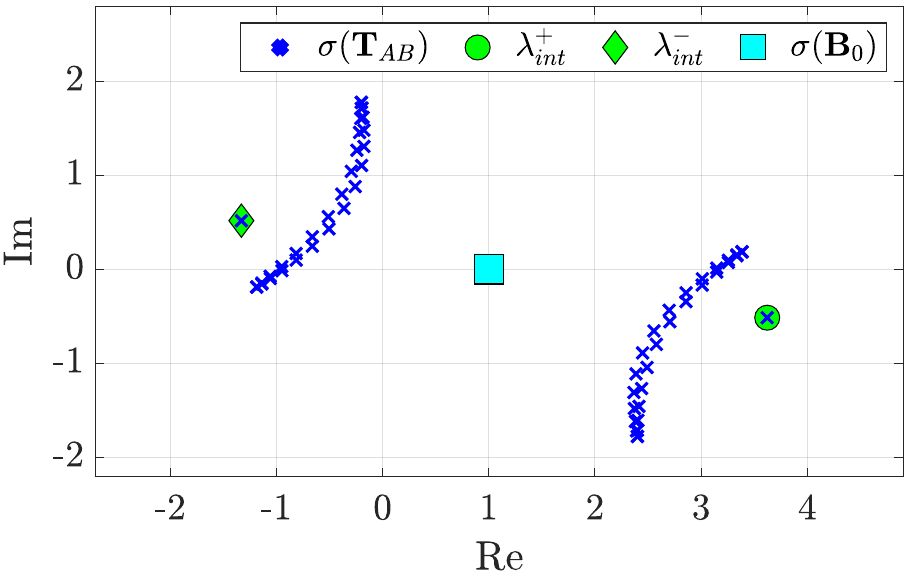}}
    \caption{Symmetric interface Toeplitz operator with complex valued entries. The interface operator supports both a monopole mode at $\lambda_{int}^-$ and a dipole mode at $\lambda_{int}^-$. The eigenvalue in $\sigma(\mathbf{B}_0)$ is not an interface mode because it is decaying towards the interface.}
    \label{Fig: Interface Twofold Dimer Common coupling}
\end{figure}

\printbibliography

\end{document}